%% file: main.tex
\documentclass[11pt,dvipsnames]{amsart}
\input{preamble}

\usepackage[margin=1in]{geometry}

\author[D. Chan]{David Chan}
\address{Department of Mathematics,
         Michigan State University,
         East Lansing, MI, 48824
         USA}
\email{chandav2@msu.edu}

\author[N. Wisdom]{Noah Wisdom}
\address{Department of Mathematics,
         Northwestern University,
         Evanston, IL, 60208
         USA}
\email{NoahAnkney2026@u.northwestern.edu}



\title{The algebraic $K$-theory of Green functors}
\date{}

\begin{document}
\begin{abstract}
    In this paper we develop computational tools to study the higher algebraic $K$-theory of Green functors.  We construct a spectral sequence converging to the algebraic $\mathbb{G}$-theory of any $G$-Green functor, for $G$ a cyclic $p$-group. From the spectral sequence we deduce a complete calculation of the algebraic $K$-theory of the constant $C_2$-Green functor associated to the field with two elements, and a calculation of the $p$-completion of the algebraic $K$-theory of the constant $G$-Green functor associated to the integers when $G$ is a cyclic $p$-group.  Additionally, we introduce the notion of a Green meadow to abstract the Green functor structure underlying clarified Tambara fields, and show, under mild conditions, that every finitely generated projective module over a $G$-Green meadow is free when $G$ is a cyclic $p$-group. This gives a computation of $K_0$ for such Green functors.
\end{abstract}
\maketitle
\tableofcontents

\input{Introduction}
\input{Background}
\input{k-modules-I}
\input{k-modules-II}
\input{spectral_sequences}
\input{applications}

\bibliographystyle{alpha}
\bibliography{ref}

\end{document}

%% file: preamble.tex
\usepackage{xy}
\usepackage{graphicx, amsmath, amssymb, amsthm,amsfonts}
\usepackage[mathscr]{eucal}
\usepackage[pagebackref, colorlinks, citecolor=Red, linkcolor=NavyBlue, urlcolor=NavyBlue]{hyperref}
\hypersetup{linktoc=all}
\usepackage{comment}
\usepackage{hyperref}
\usepackage[capitalise, nameinlink]{cleveref}
\usepackage{microtype}
\usepackage{tikz-cd}
\usepackage{spectralsequences}

\newcommand{\Q}{\mathbb{Q}}
\newcommand{\Z}{\mathbb{Z}}
\newcommand{\A}{\mathcal{A}}
\newcommand{\F}{\mathbb{F}}

\crefname{lemma}{Lemma}{Lemmas}
\crefname{theorem}{Theorem}{Theorems}
\crefname{definition}{Definition}{Definitions}
\crefname{proposition}{Proposition}{Propositions}
\crefname{remark}{Remark}{Remarks}
\crefname{corollary}{Corollary}{Corollaries}
\crefname{equation}{Equation}{Equations}
\crefname{construction}{Construction}{Constructions}
\crefname{ex}{Example}{Examples}
\crefname{example}{Example}{Example}
\crefname{appsec}{Appendix}{Appendices}
\crefname{subsection}{Subsection}{Subsections}
\crefname{letterthm}{Theorem}{Theorems}

\newtheorem{theorem}{Theorem}[section]
\newtheorem{letterthm}{Theorem}

\newtheorem{lemma}[theorem]{Lemma}

\newtheorem{proposition}[theorem]{Proposition}
\newtheorem{corollary}[theorem]{Corollary}

\newtheoremstyle{BoldRemark} 
{10pt}                    
{10pt}                    
{\upshape}                   
{}                           
{\bfseries}                  
{.}                          
{.5em}                       
{}  
\theoremstyle{BoldRemark}
\newtheorem{remark}[theorem]{Remark}
\newtheorem{construction}[theorem]{Construction}
\newtheorem{example}[theorem]{Example}
\newtheorem{definition}[theorem]{Definition}
\newtheorem{problem}[theorem]{Problem}
\newtheorem{notation}[theorem]{Notation}

\usepackage[textwidth=25mm, textsize=tiny]{todonotes}

\newcommand\restr[2]{{
		\left.\kern-\nulldelimiterspace 
		#1 
		\vphantom{\big|} 
		\right|_{#2} 
}}

\newcommand{\Res}{\mathrm{Res}}
\newcommand{\res}{\mathrm{res}}
\newcommand{\Ind}{\mathrm{Ind}}
\newcommand{\ind}{\mathrm{ind}}

\newcommand{\Hom}{\mathrm{Hom}}
\newcommand{\Tr}{\mathrm{Tr}}
\newcommand{\tr}{\mathrm{tr}}
\newcommand{\FP}{\mathrm{FP}}
\newcommand{\G}{\mathbb{G}}

\newcommand{\Mod}{\mathrm{Mod}}
\newcommand{\Mack}{\mathrm{Mack}}

\newcommand{\Set}{\mathrm{Set}}

\newcommand{\kModfg}{k\textrm{-}\Mod^{\mathrm{f.g.}}}
\newcommand{\RModfg}{R\textrm{-}\Mod^{\mathrm{f.g.}}}

\newcommand{\fg}{\mathrm{f.g.}}

\renewcommand{\A}{\mathcal{A}}

%% file: Introduction.tex
\section{Introduction}

        In this paper we develop computational tools for the algebraic $K$-theory of certain generalizations of rings, called Green functors, which play a prominent role in equivariant stable homotopy theory and representation theory.  The notion of a Green functor was introduced by Dress and axiomatizes the fundamental operations of induction, restriction, conjugation, and tensor product \cite{Dress}. The connection to stable homotopy theory comes from the fact that Green functors arise naturally as the coefficients of multiplicative (Bredon) equivariant cohomology theories. Said another way, the zeroth homotopy groups of an equivariant ring spectrum organize into a Green functor.


        Just like ordinary rings, Green functors have a notion of module, and we can seek to understand and compare different Green functors by looking at their categories of modules.  For instance, a natural question to ask is when two Green functors have equivalent categories of modules; that is, when are they Morita equivalent?  A highly successful framework for studying questions like this in the context of ordinary rings is algebraic $K$-theory, and the work in this paper extends these ideas to the study of Green functors.  

        Let $G$ be a finite group, and let $R$ be a Green functor for the group $G$. In some nice cases, we can make sense of whether or not the integer $|G|$ is a unit for $R$.  In the case that $|G|$ is invertible, Maschke's theorem implies that the category of $R$-modules is equivalent to a product of module categories over certain group rings.  This splitting is often called the Greenlees--May splitting, as the first result of this kind appears in \cite{GreenleesMay}; see \cite{BoucDellAmbrogioMartos} for the general case. It follows that the algebraic $K$-theory of $R$ is completely understood, as least insofar as we understand the algebraic $K$-theory of group rings.  
        There are several examples in the literature \cite{Thevenaz--Webb,Gre92} of classifications of projective $R$-modules for certain $G$-Green functors $R$ in which $|G|$ is not invertible. These computations lead naturally to computations of the zeroth algebraic $K$-groups of these Green functors. When $|G|$ is not invertible in $R$, the authors are not aware of any computations of the higher algebraic $K$-groups for $R$.  Our first main result gives a complete computation for $\underline{\F_2}$, the constant $C_2$-Green functor at the field with two elements.
        
        \begin{letterthm}\label{introthm: F2}
                Let $\underline{\F_2}$ denote the constant $C_{2}$-Green functor on the field $\F_2$. There is an isomorphism in the stable homotopy category
                \[
                    K(\underline{\F_2})\cong  K(\F_2)\vee K(\F_2).
                \]
            \end{letterthm}
       
        The algebraic $K$-groups of $\F_2$ were computed by Quillen in \cite{Qui72}, thus we have the following computation of the $K$-groups:
        \[
            K_{n}(\underline{\F_2}) \cong \begin{cases}
                \Z^{\oplus 2} & n=0\\
                (\mathbb{Z}/(2^i-1))^{\oplus 2} & n = 2i-1\\
                0 & \mathrm{else}.
            \end{cases}
        \]
        The Green functor $\underline{\F_2}$ plays the role of mod $2$ coefficients in genuine equivariant cohomology theories and is perhaps the most commonly used coefficient ring in equivariant homotopy theory. For applications and more discussion see \cite{DHM24,May,Hazel:fundamentalclass,BehrensWilson,HahnWilson,Petersen}.

        Just as $\underline{\F_2}$ is used for the study of mod $2$ cohomology, the role of integral coefficients is often played by a Green functor we denote by $\underline{\Z}$.  Our next main result gives a comparison between the algebraic $K$-theory of $\underline{\Z}$ and the algebraic $\G$-theory of a group ring. 

        \begin{letterthm}[\cref{thm:K-theory-of-FP(Z)}]\label{introthm: K of constant Z}
        Let $p>0$ be prime, $n\geq 1$, and let $G = C_{p^n}$ be the cyclic group of order $p^n$. There is a map of spectra
        \[ 
            K(\underline{\Z}) \rightarrow \G(\Z[C_{p^n}])
        \] 
        which becomes an equivalence 
        \[ 
            \tau_{\geq 2} K(\underline{\Z})_p^\wedge \cong \tau_{\geq 2} \G(\Z[C_{p^n}])_p^\wedge 
        \]
        upon $p$-completing and taking $2$-connective covers.
    \end{letterthm}

    For a Noetherian ring $R$, the $\G$-theory spectrum $\G(R)$ is the algebraic $K$-theory of the abelian category of finitely generated left $R$-modules. When $R = \Z[C_{p^n}]$, a theorem of Webb gives a decomposition of the groups $\G_n(\Z[G])$ in terms of the $K$-theory of certain localizations of rings of integers \cite[(0.3)]{Webb:nilpotent}. For instance, if we work over the group $G = C_p$, there are isomorphisms
    \[
        \G_n(\Z[C_{p}])\cong K_i(\Z)\oplus K_i(\Z[\zeta_p,p^{-1}])
    \]
    where $\zeta_p$ is a primitive $p$-th root of unity. When $i$ is odd, these groups are entirely known (see \cite[Theorem 1]{Weibel:survey}), hence the $p$-completion of $K_i(\underline{\Z})$ is completely computed for odd $i \geq 2$.
    
    We note that the category of $\underline{\Z}$-modules is equivalent to the category of \emph{cohomological Mackey functors}, which have been studied extensively in representation theory. In \cite{Yos83}, Yoshida proves that the category of cohomological $G$-Mackey functors is equivalent to the category of modules over the Hecke algebra
    \[
        \mathcal{H}(G) := \mathrm{End}_{\Z[G]}\left( \bigoplus_{H \subset G} \Z[G/H] \right) \mathrm{.}
    \]
    From this perspective, all of our results on the $\G$-theory and $K$-theory of Green functors of the form $\underline{\Z}$ are also results on the relevant Hecke algebras. There are also connections between cohomological Mackey functors and Artin motives; see \cite{BalmerGallauer} for details. 

Finally, we note that the results above can be viewed as the computations of the algebraic $K$-theory of certain genuine equivariant $G$-ring spectra. Indeed, any Green functor $R$ determines an Eilenberg MacLane spectrum $HR$, and the algebraic $K$-theory of $R$ and $HR$ agree, essentially by the results of \cite{DuggerShipley}.  The algebraic $K$-theory of equivariant ring spectra has seen significant recent interest and our results give some of the first explicit computations \cite{Elmantohaugseng,CnossenHaugsengLenzLinskins:Tambara,MalkiewichMerling1}. To make this claim precise, in \cite[\S 4.3]{Elmantohaugseng}, it is shown that there is a genuine $G$-spectrum $K_G(E)$ for any genuine $G$-ring spectrum $E$.  In the case where $E = HR$ for a Green functor $R$, there is an equivalence of spectra $K(R)\simeq K_G(HR)^G$, where the latter is the (genuine) $G$-fixed points of the $G$-spectrum $K_G(HR)$.  We note that $K_G$ is \emph{not} the same as the equivariant $K$-theory construction considered by Merling and Barwick \cite{Mer17,Barwick}.
\vspace{3mm}

\noindent \textbf{Methods.}  The $K$-theory computations above are achieved by performing computations of $\G$-theory groups and then relating them to algebraic $K$-groups.  Explicitly, a theorem of Dugger--Hazel--May shows that that the category of $\underline{\F_2}$-modules has finite global projective dimension, meaning that every $\underline{\F_2}$-module has a finite projective resolution \cite{DHM24}.  Applying Quillen's resolution theorem then shows $K(\underline{\F_2})\simeq \G(\underline{\F_2})$.  A theorem of Arnold \cite{Arnold} (see also \cite{BSW17}) shows that $\underline{\Z}$ also has finite global projective dimension.

Thus the $K$-theory computations above are the result of computations in algebraic $\G$-theory which apply more generally.  Our main tool is a spectral sequence with input given by the $\G$-theory of certain twisted group rings.  

\begin{letterthm}[\cref{cor:G-theory-SS}]\label{introthm: G theory SS}
    Let $R$ be a module-Noetherian $C_{p^n}$-Green functor. There is a strongly convergent spectral sequence with signature 
    \[
        E^1_{s,t} \cong \G_s(R_t) \Rightarrow \G_{s}(R)
    \] 
    where $R_t$ is a ring, depending on $t$, defined explicitly from the data of $R$. The spectral sequence collapses at the $E_{n+1}$ page.
\end{letterthm}

The spectral sequence is derived from a filtration, analogous to the isotropy filtration of a $G$-space, on the category of modules over the Green functor $R$. We note that the idea to use isotropy to filter $R$-modules is far from new; see, for instance, \cite{AyalaMazelGeeRozenbluym,GreenleesMay,BadziochDorabiala}, among many others. 

The usability of the spectral sequence arising from this filtration is a consequence of the identification of $E_1$-page. Moreover, we are able to identify conditions, satisfied in many cases of interest, under which the spectral sequence collapses at the $E_1$-page. Under these conditions we prove that in the filtration defining our spectral sequence, each map between filtered pieces is a summand inclusion, so that all extension problems are trivial. This induces a splitting of $\G$-theory spectra which we call the \emph{Greenlees--May} splitting.

\begin{letterthm}[\cref{thm:G-theory-splitting}]\label{letterthm:splitting}
    Let $R$ be a Green functor. Under certain technical conditions we have an equivalence 
    \[ 
        \G(R) \simeq \bigoplus\limits_{t=0}^n \G(R_t) 
    \] 
    of $\G$-theory spectra, where $R_t$ is a certain ring determined by $R$.
\end{letterthm}

\begin{remark}
    The Burnside Green functor mentioned in the previous theorem is the unit for the category of $G$-\emph{Mackey functors}, a generalization of $G$-modules which are to Green functors as abelian groups are to rings. In the case of the Burnside ring the Greenlees--May splitting implied by \cref{letterthm:splitting} is
    \[
        \G(\Mack_{C_{p^n}})\simeq \bigoplus_{i=0}^n \G(\Z[C_{p^i}]).
    \]
    Passing to $\pi_0$ and setting $n=1$ above, we recover a result of Greenlees  \cite{Gre92}.

    Greenlees also proves a similar splitting for the $K$-theory of $C_p$-Mackey functors and it is natural to wonder whether or not such a splitting might hold for the $K$-theory of $C_{p^n}$-Mackey functors for $n>1$. We are unable to establish this result. We note, however, that the conditions on a Green functor $R$ from \cref{letterthm:splitting} are insufficient to produce a splitting of the form in \cref{letterthm:splitting} with $K$-theory in place of $\G$-theory. Indeed, such a splitting would be in conflict with the computation of $K(\underline{\F_2})$ above.
\end{remark}

In the course of constructing the spectral sequence for $\G$-theory we prove some results about the category of modules over a certain $C_{p^n}$-Green functor which should be of interest. Originally motivated by field-like Tambara functors, our arguments more naturally applied to the following class of Green functors.

\begin{definition}\label{defintion: Green meadow}
    If $k$ is a Green functor such that each $k(G/H)$ is a field, we call $k$ a \emph{Green meadow}.
\end{definition}

We note that the class of Green meadows is distinct from the class of \emph{Mackey fields}. Importantly, the constant $G$-Green functor $\underline{\F_2}$ is a Green meadow, but not a Mackey field unless $2$ is coprime to $|G|$.

If $k$ is a Green meadow, then the Green functor structure maps make $k(G/H)$ into a vector space over $k(G/G)$. If the dimension of $k(G/H)$ over $k(G/G)$ is finite, we say that $k$ is \emph{relatively finite dimensional}. This is a special case of a more general definition for Green functors, see \cref{def:RFD}.

\begin{letterthm}[\cref{thm:proj-implies-free-for-RFD-Green-meadows}]\label{letterthm: projective implies free}
    Let $k$ be a relatively finite dimensional $C_{p^n}$-Green meadow. A $k$-module is projective if and only if it is free.
\end{letterthm}

We are able to remove the relatively finite dimensional assumption in special cases, so it does not seem that this is an essential assumption, although it simplifies arguments tremendously.

\begin{remark}
    The class of Green meadows also includes the class of \emph{clarified Tambara fields} considered in \cite{Wis24}. In conjunction with {\cite[Theorem 1.2]{Wis24}}, {\cite[Corollary B]{Wis25}}, we note that \cref{letterthm: projective implies free} applies to every relatively finite dimensional Tambara field, even those which are not clarified (using \cite[Proposition 5.7]{Wis25} and the fact that induction preserves free modules). In fact, most of our results on modules over Green meadows also apply to modules over Tambara fields straightforwardly via the results of \cite{Wis24, Wis25},
\end{remark}

The term ``free'' needs to be contextualized. The free modules over a Green functor are not graded on the integers, but rather on finite $G$-sets.  Importantly, if $X$ and $Y$ are two $G$-sets with the same cardinality but non-isomorphic $G$-actions, then the free modules on $X$ and $Y$ may or may not be isomorphic
    
A corollary of \cref{letterthm: projective implies free}  is that free modules give a set of generators for the algebraic $K_0$-group of a Green meadow $k$. In particular, there is a surjective ring homomorphism $A(C_{p^n})\to K_0(k)$, where $A(C_{p^n})$ is the Burnside ring of virtual finite $C_{p^n}$-sets. In certain cases we compute the kernel of this map, hence compute $K_0(k)$.

    \begin{letterthm}\label{intro theorem: K_0 of clarified fields}
        Let $L$ be a field with action by the group $C_{p^n}$, and let $k$ be the fixed point $C_{p^n}$-Green functor associated to $L$. The kernel of the surjection $A(C_{p^n})\to K_0(k)$ is given explicitly in \cref{corollary: free K0 for fixed point}.
    \end{letterthm}

\vspace{1mm}

\noindent \textbf{Organization:} The paper is organized as follows.  In \cref{section: background} we review the requisite background on Green functors. In \cref{section: modules over Green functors} we review some material about modules over Green functors and give a careful analysis of free modules over certain $C_{p^n}$-Green functors.  This section also contains some technical discussion of Noetherian conditions on Green functors and faithfully flat base change.  In \cref{section: K0} we prove \cref{letterthm: projective implies free,intro theorem: K_0 of clarified fields}. In \cref{section: G theory} we construct the $\G$-theory spectral sequence of \cref{introthm: G theory SS} and prove \cref{letterthm:splitting}.  Finally, the computations of \cref{introthm: F2,introthm: K of constant Z} are given in \cref{section: computations}.

\vspace{3mm}

\noindent \textbf{Acknowledgments:}  The authors thank Ben Antieau, Maxine Calle, Teena Gerhardt, John Greenlees, Mike Hill, Liam Keenan,  Mike Mandell, Andres Mejia, Maximilien P\'eroux, Noah Riggenbach, Maru Sarazola, Chase Vogeli, and Yuanning Zhang, for helpful conversations regarding the content of this paper.

The authors would also like to thank the Isaac Newton Institute for Mathematical Sciences, Cambridge, for support and hospitality during the programme ``Equivariant homotopy theory in context" where work on this paper was undertaken. This work was supported by EPSRC grant no EP/Z000580/1. The first-named author was partially supported by NSF grant DMS-2135960.

%% file: Background.tex
\section{Review of equivariant algebra}\label{section: background}

This section contains background material on Mackey and Green functors.
\subsection{Mackey functors}

We being by recalling the definition of Mackey functors. For a more thorough introduction to this topic the reader is recommended to \cite{Thevenaz--Webb}. Let $G$ be a finite group.

\begin{definition}[{cf. \cite[\S 2]{Thevenaz--Webb}}]\label{proposition: useful definition of Mackey functor}
    A $G$-Mackey functor $M$ consists of the following data.
    \begin{enumerate}
        \item for each subgroup $H \subset G$, an abelian group $M(G/H)$,
        \item for each $g \in G$, an isomorphism $c_g \colon M(G/H) \rightarrow M(G/(gHg^{-1}))$ which is the identity if $g \in H$, and such that $c_{gh} = c_{g}\circ c_{h}$
        \item for a subgroup inclusion $K \subset H$, a restriction morphism $\Res_K^H \colon M(G/H) \rightarrow M(G/K)$, such that $\Res_{g K g^{-1}}^{gHg^{-1}} c_g = c_g \Res_K^H$,
        \item for a subgroup inclusion $K \subset H$, a transfer morphism $\Tr_K^H \colon M(G/K) \rightarrow M(G/H)$, such that $\Tr_{g K g^{-1}}^{gHg^{-1}} c_g = c_g \Tr_K^H$,
        \item for any chain of subgroups $L\leq H\leq K$ then $\Res^K_L = \Res^H_L\circ \Res^K_H$ and $\Tr^K_L = \Tr^K_H\circ \Tr^{H}_L$,
        \item restrictions and transfers satisfy the double coset formula: \[ \Res_L^H \Tr_K^H = \sum\limits_{g \in L \backslash H / K} \Tr_{L \cap gKg^{-1}}^L c_g \Res_{g^{-1} L g \cap K}^K \]
    \end{enumerate}
    A morphism of Mackey functors $f \colon M \to N$ consists of a collection of group homomorphisms $f_{G/H} \colon M(G/H) \to N(G/H)$ which commute with transfers, restrictions, and conjugations. We write $\Mack^G$ for the category of $G$-Mackey functors.
\end{definition}

\begin{remark}\label{remark: Weyl group actions}
    If $g\in N_G(H)$, the normalizer of $H\leq G$, then $c_g\colon M(G/H)\to M(G/H)$ gives an action of $N_G(H)$ on $M(G/H)$.  By condition (2) in the definition this descends to action of the \emph{Weyl group} $W_G(H) = N_G(H)/H$ on $M(G/H)$ for all $H\leq G$. When the group $G$ is abelian, as it usually is in this paper, the Weyl group actions specify all the conjugation data in a $G$-Mackey functor.
\end{remark}

\begin{remark}
    A Mackey functor can also be defined as a product preserving functor $T\colon \mathcal{B}^G\to \Set$ where $\mathcal{B}^G$ is the \emph{Burnside category}. This observation is originally due to Lindner \cite{Lindner}.  The Burnside category has objects given by all finite $G$-sets and products given by disjoint union.  In some places we will evaluate a Mackey functor at a finite $G$-set $X$ which is not transitive.  Implicitly, one should pick an isomorphism $X\cong \amalg_{i=1}^n G/H_i$ which induces an isomorphism
    \[
        M(X)\cong \bigoplus\limits_{i=1}^n M(G/H_i)
    \]
    for any Mackey functor $M$.
\end{remark}

Mackey functors were originally defined by Dress \cite{Dress} to axiomatize the type of structure found in representation rings.  One can think of the restriction morphisms as ``forgetting'' a $K$-action to an $H$-action and the transfer morphisms as a form of induction. 

We consider some examples.
\begin{example}\label{example: fixed point Mackey}
    Given any $G$-action on an abelian group $M$, we may define the fixed-point Mackey functor $\mathrm{FP}(M)$ by $\mathrm{FP}(M)(G/H)= M^H$. The restriction, transfer, and conjugation data is given by the following:
    \begin{enumerate}
        \item if $K\leq H$ the restriction map $\Res^H_K\colon M^H\to M^K$ is given by the inclusion of fixed point sets,
        \item if $K\leq H$ the transfer map $\Tr^H_K\colon M^K\to M^H$ is defined on $x\in M^K$ by \[\Tr^H_K(x) = \sum\limits_{\gamma\in H/K}\gamma\cdot x.\]
        \item if $g\in G$ the conjugation isomorphism $c_g\colon M^H\to M^{gHg^{-1}}$ is the isomorphism given by $x\mapsto g\cdot x$.
    \end{enumerate}

    In particular, given any abelian group $C$, we may give it the trivial $G$-action and the resulting Mackey functor $\underline{C}$ is called the \emph{constant Mackey functor associated to $N$}. Since the group action is trivial, we have $\underline{C}(G/H) = C$, and the transfer map associated to $K \subset H$ is just multiplication by $|H/K|$.  The restriction and conjugation maps are all identities.
\end{example}

\begin{example}\label{example: burnside}
    The analogue of the integers in the category of Mackey functors is the \emph{Burnside Mackey functor} $\A_G$. It is specified by sending $G/H$ to the Grothendieck group completion of the commutative monoid of isomorphism classes of finite $H$-sets (with respect to disjoint union). For $K \subset H \subset G$, the restriction map is induced by restricting an $H$-set to a $K$-set, and the transfer is induced by the induction from a $K$-set to an $H$-set, that is, $X \mapsto H \times_K X$.  The conjugation maps are defined by identifying $H$-sets and $(gHg^{-1})$-sets for $H\leq G$ and $g\in G$.
\end{example}

\begin{example}\label{ex:rep-ring-Mackey-example}
    For any field $\mathbb{F}$, the assignment sending $G/H$ to the representation ring of $H$ over $\mathbb{F}$ forms a Mackey functor $\mathrm{RU}_{\F}$. Restrictions and transfers are defined via restriction and induction of representations. Additionally, the functor sending an $H$-set $X$ to the $\F$-vector space with basis $X$ and $H$-action obtained by permuting elements of the basis defines a morphism from the Burnside Mackey functor to this $\mathbb{F}$-representation ring Mackey functor.
\end{example}

When $G = C_p$, we pictorially represent a $C_p$-Mackey functor $M$ via a \emph{Lewis diagram}:
\[ 
    \begin{tikzcd}
    M(C_p/C_p) \arrow[d, "r"] \\
    M(C_p/e) \arrow[u, bend left = 55, "t"] \arrow[loop below, "\sigma"]
    \end{tikzcd} 
\] 
where $\sigma$ is the action by a chosen generator for $C_p$, $t$ is the only non-identity transfer, and $r$ is the  only non-identity restriction.

For instance, if $M$ is an abelian group, the constant Mackey functor associated to $M$ has the Lewis diagram
\[ 
    \begin{tikzcd}
    M \arrow[d, "\mathrm{id}_M"] \\
    M \arrow[u, bend left = 55, "p\cdot"]
    \end{tikzcd} 
\] 
where we omit the arrow $\sigma$ since the action is trivial.

\begin{example}
    The Burnside Mackey functor $\A_{C_p}$ for the group $G = C_p$ has the Lewis diagram
    \[ 
        \begin{tikzcd}[column sep = large]
            \mathbb{Z}^2 \ar[d, "r"] \\
            \mathbb{Z} \ar[u, bend left = 55, "t"]
        \end{tikzcd} 
\] 
where $r(x,y) = x+py$ and $t(z) = (0,z)$.
\end{example}

We now recall a bit more about the structure of the category of Mackey functors.  
\begin{proposition}[{cf. \cite[Proposition 3.1]{Thevenaz--Webb}}]
    The category of $G$-Mackey functors is an abelian category with enough projectives.
\end{proposition}
\begin{proof}
    The category of $G$-Mackey functors is equivalent to the category of left modules over a certain ring, called the Mackey algebra, hence is abelian with enough projectives.
\end{proof}

The category of Mackey functors has a symmetric monoidal product $\boxtimes$, called the box product, for which the unit is the Burnside Mackey functor.  The box product is constructed using Day convolution.  There is a convenient reformulation of morphisms out of a box product, called a \emph{Dress pairing}.

\begin{proposition}[{\cite[Lemma 2.17]{Shulman:Thesis}}]
\label{prop:Dress-pairing-description-of-box-product}
    The Mackey functor $M \boxtimes N$ is characterized by the following universal property. A map $M \boxtimes N \rightarrow L$ is precisely the data of a \emph{Dress pairing}: for each $H \subset G$, an abelian group homomorphism $f_H \colon M(G/H) \otimes N(G/H) \rightarrow L(G/H)$, satisfying 
    \begin{enumerate}
        \item $\Res_H^K \circ f_H = f_K \circ \left( \Res_H^K \otimes \Res_H^K \right)$,
        \item for any $g\in G$ we have $c_g \circ f_H = f_{gHg^{-1}} \circ \left( c_g \otimes c_g \right)$,
        \item $\mathrm{Tr}_H^K \circ f_H \circ \left( \mathrm{id} \otimes \Res_H^K \right) = f_K \circ \left( \mathrm{Tr}_H^K \otimes \mathrm{id} \right)$, and
        \item $\mathrm{Tr}_H^K \circ f_H \circ \left( \Res_H^K \otimes \mathrm{id} \right) = f_K \circ \left( \mathrm{id} \otimes \mathrm{Tr}_H^K \right)$.
    \end{enumerate}
\end{proposition}

\subsection{Green functors}

In this section we recall the definition of Green functors, which are analogous to rings.

\begin{definition}\label{definition: Green functor}
    A (commutative) \emph{Green functor} is a commutative monoid in the symmetric monoidal category $(\Mack^G,\boxtimes,\A_G)$.
\end{definition}
\begin{remark}
    Of course, one could consider Green functors which are not commutative, but we will not make use of these in this paper.  
\end{remark}

A Green functor is a Mackey functor $R$ together with a multiplication map $\mu\colon R\boxtimes R\to R$ and a unit map $\eta\colon \A_G \to R$ satisfying the usual associativity, commutativity, and unitality diagrams. By \cref{prop:Dress-pairing-description-of-box-product}, the map $\mu$ is equivalent to a collection of group homomorphisms
\[
    \mu_H\colon R(G/H)\otimes R(G/H)\to R(G/H)
\]
subjects to compatibility with restriction, conjugation, and transfer.  We unpack the data of a Green functor in the following lemma, which is just a specialization of \cref{prop:Dress-pairing-description-of-box-product}.
\begin{lemma}\label{lemma: Green functor data}
    A Green functor consists of a Mackey functor $R$ together with choices of commutative ring structures on $R(G/H)$ for all $H \leq G$ subject to the following conditions:
    \begin{enumerate}
        \item the restriction and conjugation maps are ring homomorphisms,
        \item the transfer map satisfies the Frobenius reciprocity relations:
        \[
            \Tr^K_H(x)\cdot y = \Tr^K_H(x\cdot \Res^K_H(y)) \quad \mathrm{and} \quad x\cdot \Tr^K_H(y) = \Tr^K_H(\Res^K_H(x)\cdot y)
        \]
        whenever these make sense.
    \end{enumerate}
\end{lemma}
Note, in particular, that the rings $R(G/H)$ have an action by the Weyl groups $W_G(H) = N_G(H)/H$ through ring automorphisms.

\begin{remark}\label{remark: Frobenius reciprocity remark}
    If $R$ is a Green functor then for any $H\leq K$ we have a ring map $\Res^K_H\colon R(G/K)\to R(G/H)$ which turns $R(G/H)$ into an $R(G/K)$-bimodule via restriction of scalars. The Frobenius reciprocity formulas amount to the statement that the transfer map
    \[
        \Tr^K_H\colon R(G/H)\to R(G/K)
    \]
    is a map of $R(G/K)$-bimodules. 
\end{remark}

A (left) module over a Green functor $R$ is a Mackey functor $M$ together with a unital and associative map $\alpha\colon R\boxtimes M\to M$.  Again, we can unpack this definition using \cref{prop:Dress-pairing-description-of-box-product}.

\begin{lemma}\label{lemma: definition of a module}
    A module over a Green functor $R$ consists of a Mackey functor $M$ together with choices of module structure maps $R(G/H)\otimes M(G/H)\to M(G/H)$ for all $H\leq G$, subject to the compatibility conditions:
    \begin{enumerate}
        \item for any $H\leq K$ we have $\Res^K_H(rm) = \Res^K_H(r) \cdot \Res^K_H(m)$,
        \item for any $H\leq G$ and $g\in G$ we have $c_g(rm) = c_g(r)c_g(m)$,
        \item Frobenius reciprocity relations, identical to those in \cref{lemma: Green functor data}, hold.
    \end{enumerate}
\end{lemma}

For a subgroup $H\leq G$, condition (2) of \cref{lemma: definition of a module} tells us the action of the Weyl group $W_G(H)$ on $M(G/H)$ is nicely compatible with the action of $W_G(H)$ on $R(G/H)$.  It is worth emphasizing that the action of $W_G(H)$ on $M(G/H)$ is \emph{not} $R(G/H)$-linear.  Instead, for any $w\in W_GH$ we have $w(rm) = w(r)w(m)$, which is sometimes referred to as a semi-linear action.  The category of $R$-modules with $W_G(H)$-semilinear action admits a nice reformulation which we now recall.

\begin{definition}\label{definition: twisted group ring}
    For a group $W$ acting on a ring $R$ through a group homomorphism $\theta\colon W\to \mathrm{Aut}(R)$, the \emph{twisted group ring} of $R$ by $W$, denoted $R_{\theta}[W]$, is the ring with the same underlying abelian group as the ordinary group ring $R[W]$, but with multiplication given by
    \[
        (r_1w_1)\cdot (r_2w_2) = r_1(\theta(w_1)(r_2))w_1w_2.
    \]
    When it is clear from context we will omit the $\theta$ from the above formula and write $w_1\cdot r_2$.  
\end{definition}

If $\theta$ is the trivial action then the twisted group ring is just the ordinary group ring.  For any $\theta$  there is a ring map $R\to R_{\theta}[W]$ which turns $R_{\theta}[W]$ into an associative $R$-algebra.

\begin{proposition}[{cf. \cite[Observation 4.3]{Mer17}}]
    For an $R$-module $M$ the following are equivalent:
    \begin{enumerate}
    \item $M$ has semilinear $W$-action.
    \item $M$ is a module over $R_{\theta}[W]$ with $R$-module structure coming from restriction of scalars.
    \end{enumerate}
\end{proposition}

\begin{corollary}
    If $R$ is a Green functor and $M$ is a module over $R$, then for all $H\leq G$, $M(G/H)$ is a module over the twisted group ring $R(G/H)_{\theta}[W_G(H)]$.
\end{corollary}

As an application, we see that fixed point Mackey functors give another good class of examples of modules.

\begin{lemma}\label{lemma: fixed point of twisted module is k-module}
    Let $k$ be a Green functor.  If $M$ is a $k(G/e)_{\theta}[G]$-module, then $\FP(M)$ is naturally a $k$-module.
\end{lemma}

\begin{proof}
    First we show $\FP(M)$ is naturally a module over $R = \FP(k(G/e))$. The functor $\FP$ is right adjoint to evaluation at $G/e$, $M \mapsto M(G/e)$, which is strong symmetric monoidal with respect to the box product. Note that the ring structure on the additive group underlying $k(G/e)$ makes it into a monoid for the tensor product of abelian groups with $G$-action, and a module over this monoid is precisely a semilinear $G$-action on a $k(G/e)$-module.
    
    Thus $\FP$ is lax symmetric monoidal and sends $k(G/e)_\theta[G]$-modules to $R$-modules, so that $M$ is an $R$-module. Now there is an adjunction unit $k \rightarrow R$, and we may view $\FP(M)$ as a $k$-module by restricting the $R$-module structure along this map.
\end{proof}

We end this section we several examples of Green functors.

\begin{example}
    The representation Mackey functor $\mathrm{RU}_\mathbb{F}$ of \cref{ex:rep-ring-Mackey-example} is a Green functor. The multiplicative structure is obtained via tensor product of representations.
\end{example}

\begin{example}
    Let $F \subset L$ be a finite Galois extension with Galois group $G$. For any $H \leq G$ we write $L^H$ for the fixed field. The assignment $G/H \mapsto L^H$ is a Green functor with restriction given by inclusion of fixed fields, transfer given by the Galois trace. We denote this Green functor by $\FP(L)$.  Note that this example can be extended to any ring $R$ with an action by a finite group $G$ through ring automorphisms.
\end{example}

\begin{example}
    For a finite group $K$ we write $H^*(K)$ for group cohomology of $K$ with integral coefficients. We write $H^{even}(K)$ for
    \[
        H^{even}(K) = \bigoplus\limits_{n=1}^{\infty} H^{2n}(K;\Z),
    \]
    the ring of even cohomology groups.  Since even cohomology classes actually commute (not just up to sign), this is a commutative ring and the assignment $G/K\mapsto H^{even}(K)$ is a Green functor.  The restriction and transfer maps are given by restriction and transfer in group cohomology.  
\end{example}

\begin{example}
    For a finite group $G$, let $E$ be an $\mathbb{E}_{\infty}$ ring in genuine $G$-spectra. For any $H\leq G$ we write $E^H$ for the (genuine) fixed points of $E$.  Then the assignment $G/H\mapsto \pi_0(E^H)$ assembles into a Green functor.  For generally, the $RO(G)$-graded homotopy groups of $E$ assembnle into an $RO(G)$-graded Green functor.
\end{example}

\subsection{Restriction and induction}

Let $G$ be a finite group and $H<G$ a subgroup. Let $\Set^G$ and $\Set^H$ denote the categories of finite $G$ and $H$-sets, respectively. The forgetful functor $\res^G_H\colon\Set^G\to\Set^H$ has a left adjoint $\ind^G_H\colon \Set^H\to \Set^G$ which is defined on objects by $X\mapsto G\times_H X$. 

\begin{definition}
    For a $G$-Mackey functor $M$, the \emph{restriction} of $M$ to $H$ is the Tambara (resp. Mackey) functor $\Res^G_H(M)$ defined by
    \[
        \Res^G_H(M)(H/K) = M(G\times_H (H/K))\cong M(G/K)
    \]
    with functoriality on morphisms given by evaluating $M$ on induced morphisms.
\end{definition}

The restriction functors are strong monoidal.

\begin{lemma}[cf. {\cite[Lemma 6.6]{Chan:biincomplete}}]
    The restriction functor is strong monoidal with respect to the box product of Mackey functors.
\end{lemma}

The restriction functor has both a left and right adjoint and they are, in fact, the same functor.

\begin{definition}
    Let $H\leq G$ be finite groups.  For an $H$-Mackey functor $M$ the \emph{induction} of $M$ to $G$ is the $G$-Mackey functor $\Ind_H^G(M)$ defined by
    \[
        \Ind^G_H(M)(X)\cong M(\res^G_H(X)).
    \]
    That is, $\Ind^G_H(M) = M\circ \res^G_H$ where $\res^G_H$ is the induced functor on the Burnside categories.
\end{definition}

\begin{proposition}[{\cite[page 1871]{Thevenaz--Webb}}]
    Induction of Mackey functors is both a left and right adjoint to restriction.
\end{proposition}

\begin{corollary}
    The induction functor is lax monoidal.
\end{corollary}

\begin{proof}
    The right adjoint of a strong monoidal functor is lax monoidal.
\end{proof}
\begin{corollary}\label{corollary: induction preserves Green functors}
    If $R$ is an $H$-Green functor then $\Ind^G_H(R)$ is a $G$-Green functor. Moreover, if $M$ is a module over $R$ then $\Ind^G_H(M)$ is a module over $\Ind^G_H(R)$.
\end{corollary}

\begin{remark}
    By an unfortunate coincidence of terminology, the Green functor structure on $\Ind_H^G(R)$ is most often referred to as the \emph{coinduction} of $R$. As we will only need the Green functor structure on $\Ind_H^G R$ in a few relatively unimportant situations, we will avoid discussion of coinduction further (although see \cite{Wis25} for a discussion of the relationship between coinduction, modules, and algebraic $K$-theory).
\end{remark}

%% file: k-modules-I.tex
\section{More on modules over Green functors}\label{section: modules over Green functors}

In this section we gather together some results concerning modules over Green functors. 

\subsection{Base change functors}

Recall the characterization of modules over a Green functor given in \cref{lemma: definition of a module}. Note that a morphism of $R$-modules is, by this description, a morphism $M \to N$ of Mackey functors such that each $M(G/H) \rightarrow N(G/H)$ is a morphism of $R(G/H)$-modules.

\begin{lemma}\label{lem:ind-res-adjunction-for-k-modules}
    Induction and restriction define functors 
    \[ 
        (\Res_H^G R)\text{-}\Mod \leftrightarrows R \text{-} \Mod 
    \] 
    which are each other's left and right adjoints.
\end{lemma}

\begin{proof}
    The lemma follows from the fact that $\Res^G_H$ and $\Ind^G_H$, as functors on all Mackey functors, are left and right adjoint to one another and $\Res^G_H$ is strong monoidal.  In particular, the claim is a consequence of  the projection formula from the \emph{Grothendieck context} in the sense of \cite[Definition 4.13(ii)]{FauskHuMay}

\end{proof}

We now describe the base change adjunction. To begin, we will need a relative tensor product for modules over a Green functor.  

\begin{definition}
    Let $R$ be a Green functor and let $M$ and $N$ be two $R$-modules. We define their box product relative to $R$ by 
    \[ 
        M \boxtimes_R N := \mathrm{Coeq} \left( M \boxtimes R \boxtimes N \rightrightarrows M \boxtimes N \right) 
    \] 
    where the two maps are those describing the $R$-module structure on $M$ and $N$ respectively.
\end{definition}

When it is clear, we will omit the subscript $R$ from the notation. The relative box product over $R$ admits a universal property, also described by a Dress pairing. Namely, aa $R$-module map out of $M \boxtimes_R N$ is precisely the data of, for each $H \subset G$, $R(G/H)$-module maps out of $M(G/H) \otimes_{R(G/H)} N(G/H)$, which satisfy the three relations of Proposition \ref{prop:Dress-pairing-description-of-box-product}. 

Note that, just as for rings, if $R \to S$ is a map of Green functors, then $S$ is naturally an $R$-module.

\begin{proposition}
\label{prop:res-base-change-adjunction}
    Let $f \colon R \rightarrow S$ be a morphism of Green functors. The functor 
    \[ 
        S \boxtimes_R - \colon R \text{-} \Mod \rightarrow S \text{-} \Mod 
    \] 
    is left adjoint to restriction of scalars 
    \[ 
        f^* \colon S \text{-} \Mod \rightarrow R \text{-} \Mod \text{.} 
    \]
\end{proposition}

We end this subsection with a brief discussion of Morita equivalences which will appear later in the paper.

\begin{definition}\label{def:Morita-equiv-of-Green-functors}
    Consider a $G$-Green functor $R$ and a $G'$-Green functor $S$. We say that $R$ and $S$ are \emph{Morita equivalent} if the category of $R$-modules is equivalent to the category of $S$-modules.
\end{definition}

We emphasize that $G$ need not be equal to $G'$ here. For example, \cref{cor:nonmodular-characteristic-Morita-equiv} gives some examples where $G$ is arbitrary and $G'$ is the trivial group.

\begin{remark}
    A result of The\'evenaz and Webb (see \cite{Thevenaz--Webb}) says that the category of Mackey functors admits a natural description as the category of left modules for an associative ring called the \emph{Mackey algebra}. Viewing the Mackey algebra as a non-commutative $e$-Green functor, this may be interpreted as a Morita equivalence statement. Similarly, if $R$ is any $G$-Green functor, then there is an associative, unital ring $\mu_R$ whose category of left modules is equivalent to $R$-modules.

    If $S$ is a $G'$-Green functor and $\mu_S$ is the associated Mackey algebra, then $R$ is Morita equivalent to $S$ if and only if $\mu_R$ is Morita equivalent to $\mu_S$ by the usual meaning of Morita equivalence for associative rings. In particular, any Morita equivalence between $R$ and $S$ is automatically additive.
\end{remark}

\subsection{Free modules}

In this subsection we discuss free modules over Green functors and give some of their basic properties. We carefully examine free modules over Green meadows in the next subsection.

\begin{lemma}
\label{lem:left-adjoint-to-exact-preserves-projectives}
    Let $L$ be a left adjoint to $R$, both functors between abelian categories. If $L$ is exact, then $R$ preserves injective objects. Dually, if $R$ is exact, then $L$ preserves projective objects.
\end{lemma}

\begin{proof}
    Assume $I$ is injective. Then $\Hom(-,R(I)) \cong \Hom(L(-),I)$ is exact. Hence $R(I)$ is injective.  The dual result is similar.
\end{proof}

\begin{corollary}
\label{cor:ind-and-res-preserve-inj-and-proj}
    Let $R$ be a Green functor. Induction 
    \[ 
        (\Res_H^G R) \text{-} \Mod \rightarrow R \text{-} \Mod 
    \] 
    and restriction 
    \[ 
        R \text{-} \Mod \rightarrow (\Res_H^G R) \text{-} \Mod 
    \] 
    preserve projective and injective objects.
\end{corollary}

\begin{proof}
    Note that the target and domain categories are abelian, so  by \cref{lem:left-adjoint-to-exact-preserves-projectives,lem:ind-res-adjunction-for-k-modules} it suffices to show both functors are exact. This follows from the fact that restriction and induction are each other's adjoints, hence each preserves kernels and cokernels.
\end{proof}

\begin{definition}
    Let $R$ be any Green functor. The free $R$-module in level $G/H$ is the $R$-module $F$ which represents the functor $M \mapsto M(G/H)$ from $R$-modules to sets.
\end{definition}

A routine application of the Yoneda lemma yields the following useful proposition.

\begin{proposition}
\label{prop:free-k-mod-in-level-G/H-is-ind-of-res}
    The free $R$-module in level $G/H$ exists. It may equivalently described as either 
    \begin{enumerate}
        \item $\Ind_H^G \Res_H^G R$ with $R$-module structure induced by the unit map 
        \[ 
            R \rightarrow \Ind_H^G \Res_H^G R . 
        \]
        \item $R \boxtimes A_{G/H}$, where $A_{G/H} \cong \mathcal{B}_G(G/H,-)$ is the free Mackey functor in level $G/H$.
    \end{enumerate}
\end{proposition}

\begin{definition}
    Let $R$ be a Green functor. We say that a module is \emph{free} if it is isomorphic to a direct sum of free modules at level $G/H_i$ for some subgroups $H_i\leq G$.  We say a module is \emph{finite free} if this direct sum is finite.  A module is \emph{finitely generated} if it is a quotient of a finite free module.
\end{definition}

\begin{lemma}
\label{lem:free-implies-projective}
    Let $R$ be any Green functor. Every free $R$-module is projective. Moreover, the finitely generated projective $R$-modules are precisely those which are retracts of direct sums of finitely many free modules.
\end{lemma}

\begin{proof}
    Let $F$ be the free $R$-module in level $G/H$. Then we have $\Mod_R(F,M) \cong M(G/H)$. Since exactness is checked levelwise, $F$ is projective. The latter claim follows from the fact that if $P$ is finitely generated then the surjection $F \to P$ from a finite free module $F$ splits.
\end{proof}

We now derive several useful consequences of \cref{prop:free-k-mod-in-level-G/H-is-ind-of-res}.

\begin{proposition}\label{prop:projectives-satisfy-MRC}
    Let $R$ be a Green functor for which all restrictions in $R$ are injective. If $P$ is any projective $R$-module, then all restrictions in $P$ are injective.
\end{proposition}

\begin{proof}
    Since every projective module is a summand, a fortiori a submodule, of a free module, it suffices to establish the result for free $R$-modules. Now $\Res_H^G$ and $\Ind_H^G$ preserve the property of all restrictions being injective, so the result follows from \cref{prop:free-k-mod-in-level-G/H-is-ind-of-res}.
\end{proof}

\begin{corollary}\label{corollary: levelwise value of free modules}
    Let $F_{G/H}$ be a free $R$-module at level $G/H$.  For any subgroup $K\leq G$ we have an isomorphism of $R(G/H)$-modules
    \[
        F_{G/H}(G/K)\cong \bigoplus\limits_{\gamma\in H\backslash G/K} R(G/(K\cap \gamma^{-1}H\gamma))
    \]
\end{corollary}

\begin{proof}
    By the isomorphism $F_{G/H} \cong \Ind^G_H\Res^G_H(R)$ and the definition of induction and restriction we have
    \[
        F_{G/H}(G/K)\cong \Ind^G_H\Res^G_H(R)(G/K)\cong k(G\times_K \Res^G_K(G/H))
    \]
    and the claim follows from the formula $\res^G_K(G/H)\cong \coprod\limits_{\gamma\in H\backslash G/K} K/(K\cap \gamma^{-1}H\gamma))$.
\end{proof}

\subsection{Free modules over Green meadows}

Recall from \cref{defintion: Green meadow} that a Green meadow is a Green functor $k$ such that each $k(G/H)$ is a field. We introduce some notation to reduce clutter.

\begin{notation}\label{noation: free module notation for Cpn}
    Let $G = C_{p^n}$ for some prime $p$ and $R$ a $C_{p^n}$-Green functor. For $0 \leq i \leq n$ we write $F_i$ for the free $R$-module on one generator in level $C_{p^n}/C_{p^i}$. If $R$ is not clear from context, we may write $F_i(R)$ instead of $F_i$.
\end{notation}

The next result is just \cref{corollary: levelwise value of free modules} for the group $G=C_{p^n}$

\begin{corollary}\label{cor:levelwise-value-of-free-modules-for-C_p^n}
    Let $R$ be a $C_{p^n}$-Green functor. For any $0\leq s\leq n$ we have
    \[
        F_i(C_{p^n}/C_{p^s}) = \begin{cases}
            R(C_{p^n}/C_{p^i})^{\oplus p^{n-s}} & s\geq i\\
            R(C_{p^n}/C_{p^s})^{\oplus p^{n-i}} & i\geq s\\
        \end{cases}
    \]
\end{corollary}

Suppose that $k = \FP(L)$ where $L$ is a field with action by $C_{p^n}$.  We would like to understand the dimension of $F_i(C_{p^n}/C_{p^s})$ as vector spaces over $k(C_{p^n/C_{p^n}}) = L^{C_{p^n}}$. Suppose that $C_{p^r}$ is the stabilizer of the action of $C_{p^n}$-action on $L$. 

Then for any $0\leq s\leq n$ we have
 \begin{equation}\label{equation: levelwise values of FP(L)}
    k(C_{p^n}/C_{p^s}) = L^{C_{p^s}}  =
    \begin{cases}
        L & s\leq r\\
        L^{C_{p^s}/C_{p^r}}  &  s \geq r.
    \end{cases}
\end{equation}

For $s\geq r$ the action of $C_{p^s}/C_{p^r}$ on $L$ is faithful, so by Artin's lemma it is Galois.  Thus the field extensions $L^{C_{p^n}/C_{p^r}}\subset L^{C_{p^s}/C_{p^r}}\subset L$ are both Galois for any choice of $s\geq r$. For any $s\geq r$ we then have $\dim_{L^{C_{p^n}}} L^{C_{p^s}/C_{p^r}} = p^{n-s}$
and so we obtain
\begin{equation}\label{equation: ranks of free modules}
    \mathrm{dim}_{k(C_{p^n}/C_{p^n})}(k(C_{p^n}/C_{p^s})) = \begin{cases}
        p^{n-r} & s\leq r\\
        p^{n-s} & s\geq r
    \end{cases}
\end{equation}
which is enough to compute the dimension of $F_i(C_{p^n}/C_{p^s})$ for any $s$. 

It is possible for $F_i$ to be isomorphic, as a $k$-module, to $k^{n}$ for some values of $i$ and $n$.  We now identify precisely when this happens for $k = \FP(L)$.  

\begin{lemma}\label{lemma: surjective transfers}
    Let $k = \FP(L)$ as above. For $r< s\leq i$ the transfer maps
    \[
        \Tr^{C_{p^{s}}}_{C_{p^{s-1}}}\colon F_i(C_{p^n}/C_{p^{s-1}})\to F_i(C_{p^n}/C_{p^{s}})
    \]
    are surjective.
\end{lemma}

\begin{proof}
    The transfers in $k$ are induced by the quotient map of $G$-sets $q\colon C_{p^n}/C_{p^{s-1}}\to C_{p^n}/C_{p^{s}}$.  The transfer if $F_i$ is induced by evaluating $k$ on the map
    \[
        C_{p^n}/C_{p^{s-1}}\times C_{p^n}/C_{p^i}\xrightarrow{q\times 1} C_{p^n}/C_{p^{s}}\times C_{p^n}/C_{p^i}.
    \]
     Since $s\leq i$ this map is equivalent to the map
     \[
        \coprod_{\ell=1}^{p^{n-i}}C_{p^n}/C_{p^{s-1}} \xrightarrow{\amalg q}\coprod_{\ell=1}^{p^{n-i}}C_{p^n}/C_{p^{s}} 
     \]
     and applying $k$ we see that that the transfer if $F_i$ is the direct sum of many copies of the transfer in $k$.  The transfer in $k$ is the Galois trace for the field extension $k(C_{p^n}/C_{p^{s-1}})\subset k(C_{p^n}/C_{p^s})$.  Since $s> r$, the Weyl group actions are faithful and this extension is Galois, hence the trace map is surjective.
\end{proof}

\begin{lemma}\label{lemma: isomorphism at level i}
    Let $k = \FP(L)$, where $L$ is a field with $C_{p^n}$-action and $C_{p^r}\subset C_{p^n}$ is the stabilizer of the action.  Then for $i\geq r$ we have an isomorphism of $k(C_{p^n}/C_{p^i})_{\theta}[C_{p^n}/C_{p^i}]$-modules $F_i(C_{p^n}/C_{p^i})\cong k(C_{p^n}/C_{p^i})^{p^{n-i}}$.
\end{lemma}

\begin{proof}
    Because $i\leq r$ we have  $C_{p^n}/C_{p^i}$ acts faithfully on the field $k(C_{p^n}/C_{p^i})$. It follows that the twisted group ring $k(C_{p^n}/C_{p^i})_{\theta}[C_{p^n}/C_{p^i}]$ is Morita equivalent to the fixed field $k(C_{p^n}/C_{p^i})^{C_{p^n}/C_{p^i}} = L^{C_{p^n}}$ (cf. \cref{lem:twisted-group-ring-Morita-invariance}).  Thus modules over this ring are classified up to isomorphism by their dimension over the $L^{C_{p^n}}$.  The dimensions of the two modules considered in the statement are computed using \cref{corollary: levelwise value of free modules} and \eqref{equation: ranks of free modules}, and both are equal to $p^{2(n-i)}$
\end{proof}

\begin{proposition}\label{proposition: when free modules are isomorphic weak form}
    Suppose that $k = \FP(L)$ where $L$ is a field of characteristic $p$ with action by $C_{p^n}$ and stabilizer $C_{p^r}$. If $i \geq r$ then there is an isomorphism of $k$-modules $F_i\cong k^{p^{n-i}}$.
\end{proposition}

For the simpler case in which $k$ is a Green meadow of characteristic not equal to $p$, one should be able to establish similar results using a Morita equivalence argument and \cref{cor:nonmodular-characteristic-Morita-equiv} below.

\begin{proof}
    By \cref{lemma: isomorphism at level i}, we have an isomorphism of $k(C_{p^n}/C_{p^i})_{\theta}[C_{p^n}/C_{p^i}]$-modules $F_i(C_{p^n}/C_{p^i})\cong k(C_{p^n}/C_{p^i})^{p^{n-i}}$. Since $F_i$ is free on the $C_{p^n}/C_{p^i}$-level we can extend this to a map of $k$-modules $\varphi\colon F_{i}\to k^{p^{n-i}}$.  Using \eqref{equation: ranks of free modules} and \cref{corollary: levelwise value of free modules} we see that these $k$-modules have the same dimension over $k(C_{p^n}/C_{p^n})$ at every level. We need to show that $\varphi$ induces an isomorphism on level $C_{p^n}/C_{p^s}$ for all $0\leq s\leq n$.
    
    First, suppose that $s\geq i$.  Since $\varphi$ is an isomorphism at level $C_{p^n}/C_{p^i}$, and all restrictions in free modules are injective by \cref{prop:projectives-satisfy-MRC}, we see that $\varphi$ is injective at level $C_{p^n}/C_{p^s}$. Hence, by a dimension count, $\varphi$ is an isomorphism on levels $C_{p^n}/C_{p^s}$ for $s\geq i$. 

    Now suppose that $r\leq s\leq i$.  We will prove that $\varphi$ induces an isomorphism at level $C_{p^n}/C_{p^s}$ via downward induction, the base case $s=i$ being covered by the definition of $\varphi$. Suppose we have proved the claim for some value of $r<s<i$; we will show that the claim holds for $s-1$. The map $\varphi$ entails a diagram
\[\begin{tikzcd}
	{F_i(C_{p^n}/C_{p^s})} && {k(C_{p^n}/C_{p^s})^{p^{n-i}}} \\
	{F_i(C_{p^n}/C_{p^{s-1}})} && {k(C_{p^n}/C_{p^{s-1}})^{p^{n-i}}}
	\arrow["{\varphi_s}", from=1-1, to=1-3]
	\arrow["\cong"', from=1-1, to=1-3]
	\arrow[hook, from=1-1, to=2-1]
	\arrow["R", hook, from=1-3, to=2-3]
	\arrow[shift left=3, two heads, from=2-1, to=1-1]
	\arrow["{\varphi_{s-1}}", from=2-1, to=2-3]
	\arrow["T", shift left=3, two heads, from=2-3, to=1-3]
\end{tikzcd}\]
where the vertical maps are transfers and restrictions.  The top horizontal map is an isomorphism by induction, the restriction is injective by \cref{prop:projectives-satisfy-MRC}, and the transfers are surjective by \cref{lemma: surjective transfers}.

Because $s > r$ the extension $k(C_{p^n}/C_{p^s})\to k(C_{p^n}/C_{p^{s-1}})$ is $C_p$-Galois, and hence the kernel of $T$ is the image of $R$. Since the diagram commutes, with either choice of vertical morphisms, we see that the kernel of $T$ is contained in the image of $\varphi_{s-1}$. Thus we have a composite
\[
    F_i(C_{p^n}/C_{p^{s-1}})\xrightarrow{\varphi_{s-1}} k(C_{p^n}/C_{p^{s-1}})^{p^{n-i}}\xrightarrow{T} k(C_{p^n}/C_{p^{s}})^{p^{n-i}}
\]
which is surjective, and the image of $\varphi_{s-1}$ contains the kernel of $T$. It follows that $\varphi_{s-1}$ must be surjective.  Indeed, if $x\in k(C_{p^n}/C_{p^s})$ then there is a $y\in F_i(C_{p^n}/C_{p^s})$ so that $T(x) = T(\varphi_{s-1}(y))$.  Then $x-\varphi_{s-1}(y)$ is in the kernel of $T$, hence is in the image of $\varphi_{s-1}$, so $x$ is in the image of $\varphi_{s-1}$. Thus $\varphi_{s-1}$ is an isomorphism since it is a surjective map between $L^{C_{p^n}}$-vector spaces of the same dimension.

All that remains is to prove that $\varphi$ induces an isomorphism at level $C_{p^n}/C_{p^s}$ for $s<r$.  In this case, a dimension count shows that the dimension of $F_i(C_{p^n}/C_{p^s})$ over $L^{C_{p^n}}$ is equal to that of $F_i(C_{p^n}/C_{p^r})$.  Since the restriction map is injective it is therefore an isomorphism.  In particular, we have commuting diagrams
\[ 
    \begin{tikzcd}
	{} \\
	{F_i(C_{p^n}/C_{p^r})} && {k(C_{p^n}/C_{p^r})^{p^{n-i}}} \\
	{F_i(C_{p^n}/C_{p^{s}})} && {k(C_{p^n}/C_{p^{s}})^{p^{n-i}}}
	\arrow["{\varphi_r}", from=2-1, to=2-3]
	\arrow["\cong"', from=2-1, to=2-3]
	\arrow["\cong"{description}, hook, from=2-1, to=3-1]
	\arrow["\cong", hook, from=2-3, to=3-3]
	\arrow[shift left=3, two heads, from=3-1, to=2-1]
	\arrow["{\varphi_{s}}", from=3-1, to=3-3]
	\arrow[shift left=3, two heads, from=3-3, to=2-3]
    \end{tikzcd} 
\]
and thus $\varphi_s$ is an isomorphism.
\end{proof}

It follows that every free module over $k = \FP(L)$ is isomorphic to one of the form
\[
    k^{n_r}\oplus \bigoplus\limits_{s=0}^{r-1} F_s^{n_s}.
\]
Our next task is to show that the numbers $n_0,\dots, n_r$ uniquely determine the free module up to isomorphism.  

\begin{proposition}\label{proposition: when free modules are isomorphic strong form}
    Let $k$ be as in \cref{proposition: when free modules are isomorphic weak form}. If there is an isomorphism
    \[
        k^{n_r}\oplus \bigoplus\limits_{s=0}^{r-1} F_s^{n_s} \cong k^{m_r}\oplus \bigoplus\limits_{s=0}^{r-1} F_s^{m_s}.
    \]
    for some non-negative integers $n_0,\dots,n_r$ and $m_0,\dots,m_r$ then $n_j=m_j$ for all $j$.
\end{proposition}

\begin{proof}
    For $0\leq s \leq n$ and $0\leq i\leq r-1$  let 
    \[
        \alpha_{s,i} = \dim_{k(C_{p^n}/C_{p^n})} F_i(C_{p^n}/C_{p^s})
    \]
    and let 
    \[
         \alpha_{s,r} = \dim_{k(C_{p^n}/C_{p^n})} k(C_{p^n}/C_{p^s}).
    \]

    The matrix $(\alpha_{s,i})$ sends the vector $(n_0,\dots,n_r)$ to the vector $(d_{0},\dots, d_{n})$ where $d_i$ is the dimension of 
    \[
         \left(k^{n_r}\oplus \bigoplus\limits_{s=0}^{r-1} F_s^{n_s}\right)(C_{p^n}/C_{p^i})
    \]
    over $k(C_{p^n}/C_{p^n})$. Considering this as a map $\alpha\colon \mathbb{Z}^{r+1}\to \mathbb{Z}^{n+1}$ the proposition will be proved if we show that this map is injective.  In particular it suffices to show that the matrix $(\alpha_{s,i})$ has rank $r+1$.  It will suffice to limit the range of $s$ to $0\leq s\leq r$ and show that the resulting matrix is rank $r+1$. Since this matrix is now square we can do this by observing that the determinant is non-zero.
    
    For $i < r$, we can use \cref{cor:levelwise-value-of-free-modules-for-C_p^n,equation: levelwise values of FP(L)} to obtain
    \[
        F_i(C_{p^n}/C_{p^s}) = k_{C_{p^n}/C_{p^i}}(C_{p^n}/C_{p^s}) =
        \begin{cases}
            L^{p^{n-s}} & s\geq i\\
            L^{p^{n-i}} & s\leq i
        \end{cases}
    \] 
    and so for $i< r$ we have
    \[
        \alpha_{s,i}  = \begin{cases}
            p^{n-r}\cdot p^{n-s} & s\geq i\\
            p^{n-r}\cdot p^{n-i} & s\leq i.
        \end{cases}
    \]
    Since we only care about the determinant being non-zero we can scale the columns of $(\alpha_{s,i})$ with $0\leq i\leq r-1$ by dividing by $p^{n-r}\cdot p^{n-r}$ and the resulting matrix, which we call $\beta$, has
    \[
        \beta_{s,i} = \begin{cases}
            p^{r-s} & s\geq i\\
            p^{r-i} & s\leq i.
        \end{cases}
    \]
    for $0\leq i\leq r-1$ and $0\leq s\leq r$. The last column in $\alpha$, $i=r$, is given by
    \[
        \alpha_{s,r} = \dim_{k(C_{p^n}/C_{p^n})}k(C_{p^n}/C_{p^s}) = p^{n-r}
    \]
    by \eqref{equation: ranks of free modules}. We also scale this column to the constant column with all $1$'s.  Finally, the resulting matrix $\gamma$ is given by
    \[ \begin{pmatrix} 
    p^{r} & p^{r-1} & p^{r-2} & \dots & 1 \\
    p^{r-1} & p^{r-1} & p^{r-2} & \dots & 1 \\
    p^{r-2} & p^{r-2} & p^{r-2} & \dots & 1 \\
    p^{r-3} & p^{r-3} & p^{r-3} & \dots & 1 \\
    \vdots & \vdots & \vdots & \ddots & \vdots \\
    1 & 1 & 1& \dots & 1
    \end{pmatrix} \]

    To see that $\gamma$ has non-zero determinant, observe that subtracting the $i$-th column from the $(i-1)$st column results in an upper triangular matrix with diagonals given by $p^{r-i-1}-p^{r-i-2}\neq 0$. Since $\gamma$ was obtained from $\alpha$ by scaling columns by nonzero integers, we see that $\alpha$ has non-zero determinant as well.
\end{proof}

Let $R$ be an arbitrary $C_{p^n}$-Green functor and let $K^{\mathrm{free}}_0(R)$ be the Grothendieck ring of isomorphism classes of finitely generated free $k$-modules where addition is direct sum and multiplication is relative tensor product. If $X$ is a finite $C_{p^n}$-set, we can write
\[
    X \cong \coprod\limits_{s=0}^n (C_{p^n}/C_{p^{s}})^{\amalg a_{s}}
\]
for some non-negative integers $a_{s}$. The assignment
\[
    X\mapsto \bigoplus\limits_{s=0}^{n} F_s^{a_s}
\]
is a ring homomorphism from the the Burnside ring $A(C_{p^n})$ to $K_0^{\mathrm{free}}(R)$. This map is surjective, since $K_0^{\mathrm{free}}(R)$ is generated by the images of $C_{p^n}/C_{p^s}$. 

On the other hand, when $k = \FP(F)$ \cref{proposition: when free modules are isomorphic strong form} and \cref{proposition: when free modules are isomorphic weak form} combine to identify the kernel of this map as precisely the subgroup of $A(C_{p^n})$ generated by elements of the form 
\begin{equation}\label{equation: kernel of K0 quotient}
    [C_{p^n}/C_{p^s}]-p^{n-s}\cdot[ C_{p^n}/C_{p^n}]
\end{equation}
for all $s$ such that $\mathrm{Stab}_{C_{p^n}}(F) \subset C_{p^s}$.

\begin{corollary}\label{corollary: free K0 for fixed point}
    Let $L$ be a field with $C_{p^n}$-action.  There is an isomorphism of rings
    \[
        K_0^{\mathrm{free}}(\FP(F))\cong A(C_{p^n})/I
    \]
    where $I$ is the ideal generated by all element of the form \eqref{equation: kernel of K0 quotient} where $s$ is such that $\mathrm{Stab}_{C_{p^n}}(L)\subset C_{p^s}$.  In particular, if the action of $G$ on $L$ is trivial then $K_0^{\mathrm{free}}(\FP(L))\cong A(C_{p^n})$.
\end{corollary}

\begin{remark}
    For any group $G$ and any $G$-Green functor $R$ there is a ring homomorphism $A(G)\to K_0^{\mathrm{free}}(R)$. In general, this map is neither injective nor surjective.  
\end{remark}

\begin{problem}\label{rem:Morita-equiv-of-Green-meadows}
    \cref{corollary: free K0 for fixed point} gives an obstruction to Morita equivalence between fixed-point Green functors, in the sense of \cref{def:Morita-equiv-of-Green-functors}. On the other hand, we show in \cref{cor:nonmodular-characteristic-Morita-equiv} that if $C_p$ acts faithfully on a field $\F$, then $\FP(\F)$ is Morita equivalent to the fixed-point subfield $\F^{C_p}$. A natural question is whether it is true more generally that when $G$ acts on a field $\F$, the $G$-Green functor $\FP(\F)$ is Morita equivalent to the $\mathrm{Stab}_{G}(\F)$-Green functor $\FP(\F^G)$ associated to the trivial $\mathrm{Stab}_{G}(\F)$-action on the fixed-point subfield $\F^G$ of $\F$.
\end{problem}

\begin{example}
    Let $G =  C_{4}$, which acts on $\F_4$ via the the Frobenius.  Since every element $x\in \F_4$ satisfies $x^4=x$, the stabilizer of the action is $C_2\subset C_4$.  The Burnside ring of $C_4$ is 
    \[
        A(C_4) \cong \mathbb{Z}[x,y]/(x^2-2x,y^2-4y,xy-2y)
    \]
    where $x$ corresponds to $C_4/C_2$ and $y$ corresponds to $C_4/e$.  To give a presentation for $K_0^{\mathrm{free}}(\FP_{C_4}(\F_4))$ we adjoin the relation $x-2$.  Notice that this immediately makes the first and third relation superfluous, so we have
    \[
        K_0^{\mathrm{free}}(\FP_{C_4}(\F_4))\cong \mathbb{Z}[y]/(y^2-4y).
    \]
\end{example}

\begin{remark}
    We will show below that all projective modules over $\FP(F)$ are free when $F$ is a field with action by a group $G$.  Thus the group $K_0^{\mathrm{free}}(\FP(F))$ is equal to the $K_0$ group of the category of finitely generated projective $\FP(F)$-modules.
\end{remark}

\begin{remark}
    The ideal $I$ in \cref{corollary: free K0 for fixed point} can be identified with the kernel of the map $A(C_{p^n})\to A(\mathrm{Stab}_{C_{p^n}}(F))$ induced by restriction of $C_{p^n}$-sets.  Thus we can embed $K_{0}^{\mathrm{free}}(\FP(F))$ into the Burnside ring of $\mathrm{Stab}_{C_{p^n}}(F)$, although this map is only surjective when the stabilizer is trivial.
\end{remark}

We can upgrade \cref{corollary: free K0 for fixed point} to a statement about all Green meadows with trivial Weyl group actions.

\begin{theorem}\label{theorem: K0 of Green functors with trivial action}
    Let $k$ be a $C_{p^n}$-Green meadow, let $L = k(C_{p^n}/e)$, and suppose that $L$ has trivial action. Then base change along the canonical map $k\to \FP(L)$ induces an isomorphism on $K_0^{\mathrm{free}}$. Thus, there is an isomorphism of rings
    \[
        K_0^{\mathrm{free}}(k)\cong K_0^{\mathrm{free}}(\FP(L))\cong A(C_{p^n}).
    \]
\end{theorem}

\begin{proof}
    There is a commutative triangle
    \[
        \begin{tikzcd}
            & A(C_{p^n}) \ar[dr,"\sim"] \ar[->>,dl] &\\
            K_0^{\mathrm{free}}(k) \ar[rr] & &  K_0^{\mathrm{free}}(\FP(L))
        \end{tikzcd}
    \]
    where the diagonal maps are the canonical surjections and the bottom map is induced by base change.  The right diagonal map is an isomorphism by \cref{corollary: free K0 for fixed point}.  It follows that the left vertical map must be injective and hence is also an isomorphism.  It then follows immediately that the base change map is also an isomorphism.
\end{proof}

\subsection{Noetherian conditions on Green functors}

Classically, the category of finitely-generated modules over a ring only forms an abelian category if the ring is Noetherian. We are interested in this condition, towards the eventual goal of understanding $\G$-theory for Green functors. In particular, some of the machinery of algebraic $K$-theory which we need only works in the setting of abelian categories, so we need to establish checkable conditions on a Green functor which imply that the category of finitely-generated modules is abelian.

We being with a few lemmas.

\begin{lemma}
    Let $R$ be a Green functor. An $R$-module $M$ is finitely generated if and only if there exists elements $x_i\in M(G/H_i)$ for $i=1,\dots, n$ such that the smallest $R$-submodule of $M$ which contains all $x_i$ is $M$.
\end{lemma}

\begin{proof}
    If such a collection of elements exist, then the maps $\widehat{x}_i\colon R_{G/H_i}\to M$ which represent each $x_i$ determine a surjection 
    \[
        \bigoplus_{i=1}^n R_{G/H_i}\xrightarrow{\oplus \widehat{x}_i} M
    \]
    and hence $M$ is finitely generated. 

    Conversely, the Yoneda lemma implies that each $R_{G/H_i}$ is generated as an $R$-module by a single universal element at level $G/H_i$. Picking a surjection from a finite free module onto $M$, the images of these universal elements form a generating set.
\end{proof}


\begin{lemma}\label{lemma: levelwise finitely generated implies finitely generated}
    If $M$ is an $R$ module such that for all $H\leq G$ $M(G/H)$ is a finitely generated $R(G/H)$-module, then $M$ is a finitely generated $R$-module.
\end{lemma}

\begin{proof}
    If a module $M$ is levelwise finitely generated then the union of the generating sets for the $M(G/H)$ is a finite generating set for $M$.  
\end{proof}

The converse is false.

\begin{example}\label{example: levelwise Noetherian is not good enough}
    Let $G=C_p$ and let and consider the Green meadow
    \[
        k = \begin{tikzcd}
	    {\mathbb{F}_p(t_1^p,t_2^p,\dots)} \\
	    {\mathbb{F}_p(t_1,t_2,\dots)}
	    \arrow["{\mathrm{inc}}", shift left, from=1-1, to=2-1]
	    \arrow["0", shift left=5, from=2-1, to=1-1]
    \end{tikzcd}
    \]
    The free module $k_{C_p/e}$ is generated by a single element at level $C_p/e$ but $k_{C_p/e}(C_p/C_e) = k(C_p/e) = \mathbb{F}_p(t_1,t_2,\dots)$ is not finitely generated over $k(C_p/C_p) = \mathbb{F}_p(t_1^p,t_2^p,\dots)$. Indeed, the inclusion factors through an infinite sequence of degree $p$ field extensions obtained by adjoining the $p$th root $t_i$ of $t_i^p$. 

    In fact, even though $k_{C_{p}/e}$ is finitely generated it contains submodules which are not finitely generated. For instance, the submodule
    \[
        M = \begin{tikzcd}
	    {\mathbb{F}_p(t_1,t_2,\dots)} \\
	    {\mathbb{F}_p(t_1,t_2,\dots)}
	    \arrow["{\mathrm{id}}", shift left, from=1-1, to=2-1]
	    \arrow["0", shift left, from=2-1, to=1-1]
    \end{tikzcd}\]
    of $k_{C_p/e}$ is not finitely generated. Indeed, for any map $k_{C_p/e} \rightarrow M$, the map $k_{C_p/e}(C_p/C_p) \rightarrow M(C_p/C_p)$ must be zero, as the transfer in $k_{C_p/e}$ is surjective, but the transfer in $M$ is zero. Thus, given any morphism $f$ from a free $k$-module to $M$, its image in $M(C_p/C_p)$ is a finite $k(C_p/C_p)$-dimensional submodule of $M(C_p/C_p)$, so that $f$ cannot be surjective.

    We highlight that $k$ actually admits the structure of a (field-like) Tambara functor with norm given by the Frobenius endomorphism. One could also make a similar example work for the Green functor $\ell$ with $\ell(C_p/C_p) = \F_2$ and $\ell(C_p/e)$ its algebraic closure, although $\ell$ does not admit the structure of a Tambara functor.
\end{example}

The key point in this example was that $R(C_p/e)$ is not a finitely generated ring over $R(C_p/C_p)$. In the setting where this holds we obtain a converse.

\begin{definition}\label{def:RFD}
    We say a Green functor $R$ is \emph{relatively finite dimensional} if for all $H \leq G$, $R(G/H)$ is a finitely generated module over $R(G/G)$ (equivalently, the restriction $R(G/G) \rightarrow R(G/H)$ is finite).
\end{definition}

 We will not hesitate to impose the relatively finite dimensional assumption whenever it simplifies arguments. In fact, this does not lose too much generality, as \cref{thm:fields-are-FF} allows for faithfully flat descent arguments in many cases (for example, \cref{thm:fields-are-FF} implies that whenever $k$ is a $C_p$-Green meadow which is not relatively finite dimensional, there is a faithfully flat map from $k$ to a relatively finite dimensional $C_p$-Green meadow). For $G \neq C_p$, however, the situation is more complicated, and we decline to give a complete picture.

 We also note that many Green functors arising in nature are relatively finite dimensional. For example, the Green functor underlying any Tambara functor which is levelwise finitely generated over $\Z$ or over a fixed field is always relatively finite dimensional (one sees this straightforwardly using \cite[Lemma 3.13]{SSw24}). Additionally, forthcoming work of Sun \cite[Theorem B]{Sun25} shows that all free polynomial Tambara functors are relatively finite dimensional.

\begin{proposition}\label{proposition: finite generation for relatively finite dimensional}
    If $R$ is relatively finite dimensional, then an $R$-module $M$ is finitely generated if and only if $M(G/H)$ is a finitely generated $R(G/H)$-module for all $H\leq G$.
\end{proposition}

\begin{proof}
    \cref{lemma: levelwise finitely generated implies finitely generated} implies that the backwards direction is true for any Green functor. For the forwards direction, suppose that $M$ is finitely generated, so that it receives a surjection from a finitely generated free $R$-module
    \[
        \bigoplus_i \Ind_{H_i}^G \Res_{H_i}^G R \rightarrow M
    \]
    It therefore suffices to observe that whenever $H$ and $K$ are any two subgroups of $G$, the $R(G/H)$-module 
    \[ 
        \left( \Ind_K^G \Res_K^G R \right) \left( G/H \right) \cong R \left( \ind_K^G \res_K^G G/H \right)
    \]
    is finitely generated. The double coset formula for $G$-sets implies that, as an $R(G/H)$-module, this is isomorphic to a finite product of $R(G/K_i)$, where the $K_i$ run through subgroups $G$ which are subconjugate to $H$. These are all finitely generated $R(G/H)$-modules by the hypothesis that $R$ is relatively finite dimensional, hence their product is a finitely generated $S(G/H)$-module.
\end{proof}

\begin{definition}
    A Green functor is \emph{module-Noetherian} if every submodule of a finitely generated module is finitely generated.
\end{definition}

\begin{remark}
    As emphasized by the name, we note that the definition of module-Noetherian given here is \emph{not} equivalent to a Green functor satisfying the ascending chain condition on ideals. Indeed, \cref{example: levelwise Noetherian is not good enough} gives an example of a Green functor which satisfies the ascending chain condition on ideals but is not module-Noetherian. To see why these two notions do not agree, it is helpful to recall why these notions are the same for ordinary commutative rings. The essential point is that the ascending chain condition for ideals on a commutative ring $R$ is equivalent to the condition that every submodule of a finite free module $R^n$ is finitely generated. Since every finitely generated module is a quotient of a finite free module it then suffices to observe that the property of being a Noetherian module is preserved under taking quotients.

    Let $R$ be a Green functor. Every finitely generated $R$-module is a quotient of a finitely generated free module, but the free modules on non-trivial $G$-sets have submodules whose structure which is not dictated by the ideals of $R$. In particular, it is possible for every ideal of $R$ to be finitely generated but have submodules of free $R$-modules be non-finitely generated, as demonstrated in \cref{example: levelwise Noetherian is not good enough}. The definition of module-Noetherian given above is equivalent to the condition that every submodule of a finitely generated free module is finitely generated.
\end{remark}

The reason we care about the module-Noetherian condition is that it implies that the category of finitely generated modules is an abelian category. Indeed, in general the category of modules over an arbitrary Green functor can fail to be Noetherian as it might not be closed under kernels.  The module-Noetherian condition is equivalent to the assumption that kernels of maps between finitely generated modules are finitely generated.

From \cref{example: levelwise Noetherian is not good enough} we see that even levelwise Noetherian Green functors need not be module-Noetherian in general. If we assume further that the Green functor is relatively finite dimensional then Noetherian and levelwise Noetherian become equivalent.

\begin{corollary}\label{cor:RFD-implies-all-Noeth-conditions-are-equiv}
    Let $R$ be a Green functor which is relatively finite dimensional. The following are equivalent:
    \begin{enumerate}
        \item $R(G/G)$ is Noetherian ring,
        \item $R(G/H)$ is Noetherian ring for all $H\leq G$,
        \item $R$ is a module-Noetherian Green functor.
    \end{enumerate}
\end{corollary}
\begin{proof}
    $(1) \implies (2)$ follows from the fact that if $S$ is a Noetherian ring and $f\colon S\to T$ is a finite ring map then $T$ is Noetherian. $(2)\implies (3)$ follows straightforwardly from \cref{proposition: finite generation for relatively finite dimensional}.
    
    For $(3)\implies (1)$, first assume that $R$ is a module-Noetherian Green functor and let $I\subset R(G/G)$ be any ideal; we must show that $I$ is finitely generated. We claim there is a Green ideal $J \subset R$ such that $J(G/G) = I \subset R(G/G)$.  Granting this, since $J \subset R$ is a submodule of a finitely generated $R$-module the assumption that $R$ is Noetherian implies that $J$ is finitely generated. But then by \cref{proposition: finite generation for relatively finite dimensional} we see that $J(G/G) = I$ is a finitely generated $R(G/G)$-module.  Since $I$ was arbitrary $R(G/G)$ must be a Noetherian ring.

    Thus it remains to prove the claim.  Let $J(G/H)\subset R(G/H)$ be the $R(G/H)$-submodule generated by $\Res^G_H(I)$.  Note that we can ignore the Weyl group actions since $\Res^G_H(I)$ is fixed by all Weyl group actions. Then certainly $J$ is a levelwise submodule, closed under conjugations, restrictions, and $J(G/G) = I$.  All that remains is to check that $J$ is actually closed under transfer.  Suppose that $x\in J(G/H)$ and $H<K$.  We can write 
    \[
        x = \sum_{i=1}^n r_i\cdot \Res^G_{H_i}(x_i)
    \]
    for some $r_i\in R(G/H)$ and $x_i\in J(G/G) = I$. Thus  
    \[
     \Tr_H^K(x) = \sum_{i=1}^n \Tr^K_H(r_i\cdot \Res^G_{H_i}(x)) = \sum_{i=1}^n \Tr_H^K(r_i)\cdot \Res^G_{K}(x_i)\in J(G/K)
    \]
    where the second equality is Frobenius reciprocity.
\end{proof}

One might wonder about whether or not a Green functor $R$ being module Noetherian with $R(G/G)$ is Noetherian implies that $R$ is relatively finite dimensional. We are not aware of any examples where this fails.

\begin{problem}
    Which finite groups satisfy the property that module-Noetherian implies relatively finite dimensional?  The authors believe this to be true for $G = C_{p^n}$.
\end{problem}

Finally, we are interested in when restriction along a Green functor morphism preserves the property of being finitely generated.

\begin{definition}\label{def:module-finite}
    We say that a Green functor map $R \rightarrow S$ is \emph{module-finite} if the restriction functor
    \[ 
        S {-} \Mod \rightarrow R {-} \Mod 
    \]
    sends finitely generated $S$-modules to finitely generated $R$-modules.
\end{definition}

If $G$ is the trivial group, module-finite is equivalent to finite (that is--the target is a finitely generated module over the domain). The following is an immediate consequence of \cref{cor:RFD-implies-all-Noeth-conditions-are-equiv}.

\begin{lemma}\label{lem:sufficient-conditions-for-module-finite}
    Let $f : R \rightarrow S$ be a levelwise finite morphism of relatively finite dimensional Green functors. Then $f$ is module-finite.
\end{lemma}

\subsection{Faithfully flat field extensions}

In this subsection we make a technical observation about Green meadows which allows us to extend some of our results to more cases. We begin by generalizing, in a straightforward way, the notion of a faithfully flat ring map to Green functors.

\begin{definition}
    A map of Green functors $R \to S$ is \emph{flat} if the base change functor $S \boxtimes_R - \colon \mathrm{Mod}_R\to \mathrm{Mod}_S$ preserves monics. A morphism is \emph{faithfully flat} if it is flat and $S\boxtimes_R M \cong 0$ implies $M \cong 0$ for any $R$-module $M$.
\end{definition}

In ordinary commutative algebra, every map of fields is faithfully flat.  In general this is not true for arbitrary maps of Green meadows.  The next theorem identifies some examples of faithfully flat maps.

\begin{theorem}\label{thm:fields-are-FF}
    Let $k \rightarrow \ell$ be a morphism of $C_p$-Green meadows, and assume the transfer in $\ell$ is zero. Then $k \rightarrow \ell$ is faithfully flat.
\end{theorem}

Our assumptions imply that $C_p$ acts trivially on $\ell(C_p/e)$ and $k(C_p/e)$ and that the transfer in $k$ is trivial as well. Indeed, if the action of $C_p$ on $\ell(C_p/e)$ were non trivial it would be faithful, in which case Artin's lemma implies that the transfer composed with the restriction surjects onto the fixed field $\ell(C_p/e)^{C_p}\neq 0$.

\begin{proof}
    Let $\overline{M}$ denote the base-change of a $k$-module $M$ to an $\ell$-module. Then we have 
    \[ 
        \overline{M}(C_p/e) \cong M(C_p/e) \otimes_{k(C_p/e)} \ell(C_p/e) 
    \] 
    and 
    \[ 
        \overline{M}(C_p/C_p) \cong \left(\left( M(C_p/C_p) \otimes_{k(C_p/C_p)} \ell(C_p/C_p) \right)\bigoplus \left( M(C_p/e) \otimes_{k(C_p/e)} \ell(C_p/e) \right)_{C_p} \right)/\mathrm{FR} 
    \] 
    where $\mathrm{FR}$ is the $k(C_p/C_p)$-subspace generated by the ``Frobenius reciprocity" relations 
    \begin{enumerate}
        \item[($1$)] $\Tr(m) \otimes x \bigoplus - m \otimes \Res(x)$ for all $m \in M(C_p/e)$ and $x \in \ell(C_p/C_p)$ and 
        \item[($2$)] $m \otimes \Tr(x) \bigoplus - \Res(m) \otimes x$ for all $m \in M(C_p/C_p)$ and $x \in \ell(C_p/e)$.
    \end{enumerate}
    Our assumption that the transfer of $\ell$ vanishes allows us to replace relation ($2$) with relation 
    \begin{enumerate}
        \item[($2'$)] $0 \bigoplus \Res(m) \otimes x$ for all $m \in M(C_p/C_p)$ and $x \in \ell(C_p/e)$.
    \end{enumerate}
    
    Since $C_p$ acts trivially on $k(C_p/e)$ and $\ell(C_p/e)$, we may commute the $C_p$-orbits with the tensor $- \otimes_{k(C_p/e)} \ell(C_p/e)$. Let $X_M\subset M(C_p/e)$ denote the $k(C_p/e)$-vector subspace generated by $\Res(M(C_p/C_p))$, noting that $X_M$ is fixed by $C_p$, and choose $Y_M\subset M(C_p/e)_{C_p}$ so that we have an isomorphism of $M(C_p/e)_{C_p} \cong X_M \oplus Y_M$. Note that relation $(2')$ will imply the terms in $\overline{M}(C_p/C_p)$ which come from $X_M$ will vanish.
    
    Let $U\subset \ell(C_p/e)$  be the $k(C_p/e)$-vector subspace generated by $\Res(\ell(C_p/C_p))$, and choose $V$ so that $\ell(C_p/e) \cong U \oplus V$ as $k(C_p/e)$-vector spaces. Using relations $(1)$ and  $(2')$ we obtain 
    \[ 
        \overline{M}(C_p/C_p) \cong \left( M(C_p/C_p) \otimes_{k(C_p/C_p)} \ell(C_p/C_p) \bigoplus Y_M \otimes_{k(C_p/e)} V \right)
    \] 
    as a $k(C_p/C_p)$-vector space.

    If $N\subset M$ is a sub-$k$-module of $M$, then we may choose $Y_N \subset Y_M$. Thus $\overline{N} \subset \overline{M}$ and basechange along $k \rightarrow \ell$ preserves subobjects -- that is $k\to \ell$ is flat. If $\overline{M} = 0$, then $M(C_p/e) = 0$ by flatness of the field map $k(C_p/e) \rightarrow \ell(C_p/e)$ whence $0 = \overline{M}(C_p/C_p) = M(C_p/C_p) \otimes_{k(C_p/C_p)} \ell(C_p/C_p)$. Since $k(C_p/C_p) \rightarrow \ell(C_p/C_p)$ is also flat, we deduce $M(C_p/C_p) = 0$, whence $M = 0$.
\end{proof}

We expect some version of \cref{thm:fields-are-FF} to be true for more complicated groups, although it is not clear exactly what the hypotheses should be on the morphism $k \to \ell$. Such sufficient criteria could be the assumption that the isotropy group of the Weyl action on each level is preserved, or that a transfer $\Tr_H^K$ in $k$ is nontrivial if and only if the transfer $\Tr_H^K$ in $\ell$ is nontrivial.

\begin{corollary}\label{corollary: faithful flatness for trivial actions}
    Let $k$ be any $C_p$-Green meadow. Then $k$ admits a faithfully flat map to a relatively finite dimensional $C_p$-Green meadow.
\end{corollary}

\begin{proof}
    If the characteristic of $k$ is not $p$, or $C_p$ acts nontrivially on $k(C_p/e)$, then $k \cong \FP(k(C_p/e))$ is already relatively finite dimensional. Otherwise the transfer in $\FP(k(C_p/e))$ is zero so that \cref{thm:fields-are-FF} applies to the canonical map $k \rightarrow \FP(k(C_p/e))$.
\end{proof}

We have not considered whether or not \cref{corollary: faithful flatness for trivial actions} should hold for other kinds of Green functors. It would be very interesting to know how necessary the Green meadow hypothesis is.

%% file: k-modules-II.tex
\section{\texorpdfstring{Towards $K_0$ of $C_{p^n}$-Green meadows}{Towards K\_0 of C\_p\string^n-Tambara fields}}\label{section: K0}

In this section we classify projective modules over (relatively finite dimensional) $C_{p^n}$-Green meadows by showing they are all free. As a consequence we are able to compute $K_0(k)$, the free abelian group generated by the projective modules with direct sum as addition.  In particular, it is isomorphic to the group $K_0^{\mathrm{free}}(k)$ considered earlier.

\subsection{The level filtration}

We begin by constructing a useful filtration of the category of $C_{p^n}$-Mackey functors.

\begin{construction}
    Given any $C_{p^n}$-Mackey functor $M$, we may functorially construct a $C_{p^{n-1}}$-Mackey functor $\tau_{\geq 1}(M)$, essentially by ``throwing out'' the $C_p/e$ level.  Precisely, set 
    \[
        \tau_{\geq 1}(M)(C_{p^{n-1}}/C_{p^s}) := M(C_{p^n}/C_{p^{s+1}}).
    \]
    So $\tau_{\geq 1}(M)$ has the same top level as $M$, and the bottom level of $\tau_{\geq 1}(M)$ is the $C_{p^n}/C_p$-level of $M$. 
    
    Next, set the restriction and transfer maps between immediately adjacent levels of $\tau_{\geq 1}(M)$ to be the restriction and transfer maps between the corresponding levels of $M$. Note that to check that the double-coset formula holds, it suffices to check that it holds for immediately adjacent levels of $\tau_{\geq 1}(M)$. There, it holds because it is the same formula as for $M$. 
\end{construction} 

\begin{lemma}
	Let $k$ be a $C_{p^n}$-Green meadow. Then $\tau_{\geq 1} k$ is a $C_{p^{n-1}}$-Green meadow.
\end{lemma}


\begin{proposition}\label{prop:tau_geq_1_determines-k-mod-functors-and-has-both-adjoints}
    Let $R$ be any $C_{p^n}$-Green functor. The functor $\tau_{\geq 1}$ restricts to a functor $R \text{-} \Mod \rightarrow \tau_{\geq 1}(R) \text{-} \Mod$ which preserves all limits and colimits.
\end{proposition}

\begin{proof}
    The claim $\tau_{\geq 1}$ sends $R$ modules to $\tau_{\geq 1}(R)$-modules follows immediately from \cref{lemma: definition of a module} (alternatively, one may show that $\tau_{\geq 1}$ is lax symmetric monoidal). All limits and colimits of modules over Green functors are computed levelwise, hence are preserved by $\tau_{\geq 1}$.
\end{proof}

Next, we aim to study the effect of $\tau_{\geq 1}$ on free modules.

\begin{lemma}
\label{lem:tau-preserves-free-modules}
    Let $R$ be a $C_{p^n}$-Green functor and $1 \leq i \leq n$. Then we have an isomorphism 
    \[ 
        \tau_{\geq 1} F_i(R) \cong F_{i-1} (\tau_{\geq 1} R) 
    \] 
    of $\tau_{\geq 1} R$-modules.
\end{lemma}

\begin{proof}
    To reduce clutter, we use $\Res_i^j$ in place of the restriction along the subgroup inclusion $C_{p^i} \subset C_{p^j}$. We employ similar notation for $\Tr$. To start, we show $\tau_{\geq 1} F_i$ is generated by a single element.
    
    Let $x \in F_i(C_{p^n}/C_{p^i})$ denote a generator: the element corresponding to the identity $\mathrm{Id}_{F_i}$ under a Yoneda isomorphism. An arbitrary element $y \in F_i(C_{p^n}/e)$ is a sum of elements of the form $c_g (a \cdot \Res_0^i (x))$ where $a \in R(C_{p^n}/e)$ and $g \in C_{p^n}$, so that 
    \[ 
        \Tr_0^1(y) = \Tr_0^1((c_g a) \cdot \Res_0^i (x)) = \Tr_0^1(c_g a) \cdot \Res_1^i(x) 
    \]
    by Frobenius reciprocity. Thus whenever $j \geq 1$, each $F_i(C_{p^n}/C_{p^j})$ is generated as an $R(C_{p^n}/C_{p^j})$-module by the transfers of $R(C_{p^n}/C_{p^t})$-multiples of $\Res_t^i(x)$ for $1 \leq t \leq i$.
    
    It follows that $\tau_{\geq 1} F_i$ is generated over $\tau_{\geq 1} R$ by a single element in level $C_{p^{n-1}}/C_{p^{i-1}}$. Let $T_{i-1}$ denote the free $\tau_{\geq 1} R$-module in level $C_{p^{n-1}}/C_{p^{i-1}}$, so that we have a surjection $T_{i-1} \rightarrow \tau_{\geq 1} F_i$ by the previous paragraph. Define a $R$-module $\widetilde{T_{i-1}}$ by $\tau_{\geq 1} \widetilde{T_{i-1}} = T_{i-1}$ and 
    \[ 
        \widetilde{T_{i-1}}(C_{p^n}/e) := T_{i-1}(C_{p^{n-1}}/e) \otimes_{R(C_{p^n}/C_p)} R(C_{p^n}/e) 
    \] 
    The generator for $T_{i-1}$ is classified by a map $F_i \rightarrow \widetilde{T_{i-1}}$ which determines the second map in the composition
    \[ 
        T_{i-1} \rightarrow \tau_{\geq 1} F_i \rightarrow T_{i-1} 
    \] 
    by applying $\tau_{\geq 1}$. This map sends the generator of $T_{i-1}$ to itself, hence is the identity. Thus the first map $T_{i-1} \rightarrow \tau_{\geq 1} F_i$, which is a surjection, is also an injection, hence an isomorphism.
\end{proof}

\begin{definition}
    Let $R$ be a $C_{p^n}$-Green functor, and $M$ an $R$-module. Denote by $R \text{-} \Mod^{\geq 1}$ the full subcategory of $R$-modules whose objects are of the form 
    \[ 
        \mathrm{Coeq} \left( \bigoplus_{i\in I} F_{a_i} \rightrightarrows\bigoplus_{j\in J} F_{a_j} \right) 
    \] 
    where $1 \leq a_i, a_j \leq n$.
\end{definition}

Roughly, $R \text{-} \Mod^{\geq 1}$ is the subcategory of modules which have a presentation by generators and relations, none of which are in level $C_{p^n}/e$.

\begin{proposition}
\label{prop:equivalence-of-cats-tau}
    The functor $\tau_{\geq 1}$ induces an equivalence between the category of $\tau_{\geq 1}(R)$-modules and $R \text{-} \Mod^{\geq 1}$.
\end{proposition}

\begin{proof}
    Let $\mathcal{C}$ be the full subcategory of $R$-modules with objects the free $R$-modules on levels $C_{p^n}/C_{p^s}$ for $1 \leq s \leq n$, and let $\mathcal{D}$ be the full subcategory of $\tau_{\geq 1}(R)$-modules with objects all free $\tau_{\geq 1} R$-modules. By \cref{lem:tau-preserves-free-modules}, $\tau_{\geq 1}$ defines a functor $\mathcal{C} \rightarrow \mathcal{D}$ which is essentially surjective. This functor is fully faithful because a free module in level $X$ represents the functor sending a module to its value in level $X$. Therefore $\tau_{\geq 1}$ defines an equivalence of categories $\mathcal{C} \cong \mathcal{D}$.

    Now we show $\tau_{\geq 1}$ is fully faithful on all of $R \textrm{-} \Mod^{\geq 1}$. By \cref{prop:tau_geq_1_determines-k-mod-functors-and-has-both-adjoints} $\tau_{\geq 1}$ preserves limits and colimits. Now let 
    \[ 
        M \cong \textrm{Coeq} \left( \bigoplus_{j \in J} F_{a_j} \rightrightarrows \bigoplus_{i \in I} F_{a_i} \right)
    \] 
    be an arbitrary element of $R \text{-} \Mod^{\geq 1}$ (thus $a_i, a_j \geq 1$).  For any $R$-module $N$ we have 
    \begin{align*}
        \hom_R(M,N) & \cong \hom_R \left( \textrm{Coeq} \left( \oplus_{j \in J} F_{a_j} \rightrightarrows \oplus_{i \in I} F_{a_i} \right),N \right) \\
        & \cong \textrm{Eq} \left( \hom_R(\oplus_{i \in I} F_{a_i},N) \rightrightarrows \hom_R(\oplus_{j \in J} F_{a_j},N) \right) \\
        & \cong \textrm{Eq} \left( \hom_{\tau_{\geq 1} R}(\oplus_{i \in I} F_{a_i-1},\tau_{\geq 1} N) \rightrightarrows \hom_{\tau_{\geq 1} R}(\oplus_{j \in J} F_{a_j-1},\tau_{\geq 1} N) \right) \\
        & \cong \hom_{\tau_{\geq 1} R} \left( \textrm{Coeq} \left( \oplus_{j \in J} F_{a_j-1} \rightrightarrows \oplus_{i \in I} F_{a_i-1} \right), \tau_{\geq 1} N \right) \\
        & \cong \hom_{\tau_{\geq 1} R}(\tau_{\geq 1} M, \tau_{\geq 1}N)
    \end{align*}
    where the third isomorphism holds due to the fact that both terms in the equalizer diagram are the classifying sets of choices of elements in the groups $N(C_{p^n}/C_{p^s})$ for $s \geq 1$, and the last isomorphism follows from the preservation of coequalizers by $\tau_{\geq 1}$.
\end{proof}

Viewing $\tau_{\geq 1}$ as coinduction along the surjection $C_{p^n} \rightarrow C_{p^{n-1}}$, \cref{prop:equivalence-of-cats-tau} may be interpreted roughly as a dual to \cite[Theorem F]{Wis25} (using the fact that coinduction and induction agree on Mackey functors). 

Since $\tau_{\geq 1}$ preserves (co)limits and $\tau_{\geq 1}(k) \text{-} \Mod$ is an abelian category, we see that $k \text{-} \Mod^{\geq 1}$ is an abelian subcategory of $k \text{-} \Mod$. It is also a Serre subcategory by the Horseshoe lemma; the authors intend to study the resulting localization sequence on algebraic $K$-theory in future work.

\subsection{\texorpdfstring{Projective modules over $C_{p^n}$-Tambara fields}{Projective modules over C\_p\string^n-Tambara fields}}

We are finally able to show that every finitely generated projective module over a Tambara field is free.

\begin{theorem}\label{thm:proj-implies-free-for-RFD-Green-meadows}
	Let $k$ be a relatively finite dimensional $C_{p^n}$-Green meadow. Every finitely generated projective $k$-module is free.
\end{theorem}

The idea is to split off free module summands until nothing remains. Since injective subobjects are summands for formal reasons, we begin by establishing injectivity of a free module.

\begin{lemma}\label{lemma: F zero is injective}
    For any Green meadow $k$, the free module $F_0$ is an injective $k$-module.
\end{lemma}

\begin{proof}
    We have $F_0 = \Ind^n_0(k(C_{p^n}/e))$, where 
    \[
        \Ind^n_0\colon k(C_{p^n}/e)\textrm{-}\Mod\to k\textrm{-}\Mod
    \]
    is the induction functor. Since $k(C_{p^n}/e)$ is a field, it is injective over itself. Induction preserves injective objects by \cref{cor:ind-and-res-preserve-inj-and-proj}, so that $F_0$ must be an injective $k$-module.
\end{proof} 

\begin{proof}[Proof of \cref{thm:proj-implies-free-for-RFD-Green-meadows}]
	We induct on $n$, the exponent of the order of the group $G = C_{p^n}$.
	
	Let $P$ be a projective $k$-module. Pick a projective $k$-module $Q$ such that $P \oplus Q \cong \oplus_{i=0}^n F_i^{s_i}$. Since the restrictions in the free modules are injective we see that all the restrictions in $P$ are injective. The double coset formula therefore completely determines all transfers on $P$. If $x \in P(C_{p^n}/e)$ is an element with the property that the $C_{p^n}$-orbits of $x$ are linearly independent over $k(C_{p^n}/e)$, then $x$ generates a copy of $\left( \Ind_e^{C_{p^n}} k(C_{p^n}/e) \right) (C_{p^n}/e)$ in the bottom level of $P$. The double coset formula then implies that $x$ specifies a submodule of $P$ isomorphic to $\Ind_e^{C_{p^n}} k(C_{p^n}/e)$. Since $\Ind_e^{C_{p^n}} k(C_{p^n}/e)$ is an injective $k$-module by \cref{lemma: F zero is injective}, it splits off as a summand. Repeating this for all $x \in P(C_{p^n}/e)$ such that the $C_{p^n}$-orbits of $x$ are linearly independent over $k(C_{p^n}/e)$, we are entitled to write $P \cong P' \oplus \left( \Ind_e^{C_{p^n}} k(C_{p^n}/e) \right)^{\oplus s}$ as $k$-modules; $s$ is finite because $P$ is finitely generated.
	
	Repeating this argument for $Q$, we may write \[ F_0^{\oplus s} \oplus P' \oplus Q' \cong  \oplus_{i=0}^n F_i^{s_i} \] where $P'$ and $Q'$ have the property that, for every element $x$ in $P'(C_{p^n}/e)$ or $Q'(C_{p^n}/e)$, the $C_{p^n}$-orbits of $x$ have a linear dependence. Before we may proceed, we must observe that $s = s_0$. This follows from the fact that, for any finitely generated $k$-module $M$, the dimension of the subspace of $M(C_{p^n}/e)$ spanned by the $C_{p^n}$-orbits of those elements $x$ for which the $C_{p^n}$-orbits of $x$ are linearly independent is invariant under isomorphism. In our case, those dimensions are respectively $s p^n$ and $s_0 p^n$ and thus $s = s_0$.
    
    Next, we aim to show that $P' \oplus Q' \cong \oplus_{i=1}^n F_i^{s_i}$. Let $M = P' \oplus Q'$ to avoid notational clutter. Consider the map \[ \phi\colon M \hookrightarrow M \oplus F_0^{s_0} \cong \bigoplus_{i=0}^n F_i^{s_i} \twoheadrightarrow \bigoplus_{i=1}^n F_i^{s_i} \] which we aim to show is an isomorphism. Note that the kernel in level $C_{p^n}/e$ consists of elements of $M$ which are restrictions of transfers of elements $x \in M(C_{p^n}/e)$ such that the $C_{p^n}$-orbits of $x$ are linearly independent over $k(C_{p^n}/e)$. By construction $M$ contains no such elements, hence $\phi$ in level $C_{p^n}/e$ is injective. Since all restriction maps for $M$ are injective, $\phi$ is injective in all other levels, hence is an injective $k$-module map. Since $k$ is relatively finite dimensional each level $k(C_{p^n}/C_{p^i})$ is a finite dimensional $k(C_{p^n}/C_{p^n})$ vector space.  By dimensional counting we see that $\phi$ is an isomorphism.
	
	Now observe that $P'$ belongs to the subcategory $k {-} \Mod^{\geq 1}$ because it is the coequalizer of the diagram $\oplus_{i=1}^n F_i^{s_i} \rightrightarrows \oplus_{i=1}^n F_i^{s_i}$ where one of the arrows is zero and the other arrow is the composition $\oplus_{i=1}^n F_i^{s_i} \rightarrow Q' \rightarrow \oplus_{i=1}^n F_i^{s_i}$ of projection followed by inclusion. In particular, $\tau_{\geq 1}(P')$ is a finitely generated projective module in $\tau_{\geq 1}(k) \text{-} \Mod$. By induction, $\tau_{\geq 1}(P')$ is a direct sum of free modules, hence so is $P'$ since the equivalence $k \text{-} \Mod^{\geq 1} \cong \tau_{\geq 1}(k) \text{-} \Mod$ of \cref{prop:equivalence-of-cats-tau} restricts to a bijection between free modules. Implicitly we have used that if $k$ is relatively finite dimensional then so is $\tau_{\geq 1} k$.
\end{proof}

We immediately obtain the following.

\begin{corollary}\label{cor:free-things-generate-K_0-of-field}
	Let $k$ be a relatively finite dimensional $C_{p^n}$-Green meadow. Then $K_0(k)\cong K_0^{\mathrm{free}}(k)$.
\end{corollary}

\begin{remark}
    We note that the conclusion of \cref{thm:proj-implies-free-for-RFD-Green-meadows,cor:free-things-generate-K_0-of-field} actually hold as long as $k$ admits a faithfully flat map to a relatively finite dimensional Green meadow. In particular this always happens when $G =C_p$ by \cref{thm:fields-are-FF}.
\end{remark}

If $R$ is any Green functor, the symmetric monoidal structure on $R \text{-} \Mod$ induces a ring structure on $K_0(R)$. Using \cref{cor:free-things-generate-K_0-of-field,theorem: K0 of Green functors with trivial action} we compute $K_0(k)$ as a ring when $k$ is a $C_{p^n}$-Green meadow.  

\begin{theorem}\label{thm:K_0-of-clarified-triv-action-Cpn-Tamb-fields}
	Let $k$ be a $C_{p^n}$-Green meadow such that $C_{p^n}$ acts trivially on $k(C_{p^n}/e)$. Then $K_0(k)$ is isomorphic to the $C_{p^n}$-Burnside ring $A({C_{p^n}})$.
\end{theorem}

We conclude this section with an example of a projective, non-free module over the Burnside Green functor. We emphasize the contrast with the classical situation: for the initial ring $\Z$, every projective module is free. Therefore the fact that \cref{thm:proj-implies-free-for-RFD-Green-meadows} holds for Green meadows is a particularly nice property enjoyed by these objects. 

\begin{example}
Let $G = C_5$, and consider the \emph{twisted} Burnside functor
\[
    \widetilde{\mathcal{A}} = 
\begin{tikzcd}
	{\mathbb{Z}^2} \\
	\\
	\mathbb{Z}
	\arrow["{(2,5)}"', shift right=2, from=1-1, to=3-1]
	\arrow["\begin{array}{c} \begin{pmatrix} 0 \\ 1  \end{pmatrix} \end{array}"', shift right=2, from=3-1, to=1-1]
\end{tikzcd}
\]
which we claim is projective in the category of $C_5$-Mackey functors. To see this, it suffices to compute that $\widetilde{\mathcal{A}}\boxtimes \widetilde{\mathcal{A}}$ is isomorphic to the Burnside Mackey functor $\mathcal{A}$ defined in \cref{example: burnside}, and hence $\widetilde{\mathcal{A}}$ is an element in the Picard group. It then follows formally that it must be projective.  

To see that $\widetilde{\mathcal{A}}$ is not free, a quick count of ranks implies that it suffices to observe that it is not isomorphic to the $\mathcal{A}$. To see this, suppose there is an isomorphism $\mathcal{A}\to \widetilde{\mathcal{A}}$ could be represented by a matrix 
\[
    M = \begin{pmatrix}
            a & b \\
            c & d
        \end{pmatrix}
    \in \mathrm{Mat}_2(\mathbb{Z})
\]
with $\det(M)=\pm 1$ for the map on the $C_5/C_5$-level and $\pm1$ on the $C_5/e$-level.  Compatibility with the restrictions implies that we must have 
\[
    a+5c =2\quad b+5d=5
\]
while compatibility with the transfers implies that $b=0$ and $d=\pm1$.  We compute
\[
    \pm 1 = \det(M) = ad-bc = \pm a = \pm(2-5c)
\]
which cannot be $\pm1$ since $c$ is an integer.
\end{example}

%% file: spectral_sequences.tex
\section{\texorpdfstring{The $\G$-theory of Green functors}{The G-theory of Green functors}}\label{section: G theory}

In this section we describe some computational tools for studying the $K$-theory and $\G$-theory of Green functors. We will assume throughout that our Green functors are module-Noetherian, as we will need that the category of finitely generated $R$-modules is an abelian category.

\subsection{\texorpdfstring{A spectral sequence for $\G$-theory}{A spectral sequence for G-theory}}

For a Noetherian ring $R$, the $\G$-theory of $R$ is the $K$-theory of the abelian category of finitely generated $R$-modules. This definition extends naturally to module-Noetherian Green functors.

\begin{definition}
    Let $R$ be a module-Noetherian Green functor. Then the $\G$-theory spectrum $\G(R)$ is defined to be the $K$-theory spectrum of the abelian category $R$-$\Mod^{f.g.}$ of finitely-generated $R$-modules.
\end{definition}

Our goal is to study the homotopy groups of $\G$-thoery spectra associated to various $C_{p^n}$-Green functors. To facilitate our analysis we make use of a filtration on the category of $C_{p^n}$-Mackey functors. To begin, we say a Mackey functor $M$ is \emph{brutally truncated} if $M(C_{p^n}/e) = 0$.

\begin{definition}
    Let $M$ be a $C_{p^n}$-Mackey functor. We define its \emph{brutal truncation} by $\tau_{\geq 1}^{\mathrm{brutal}} M$ by 
    \[ 
        \tau_{\geq 1}^{\mathrm{brutal}} M(C_{p^n}/e) = 0 
    \] 
    and
    \[ 
        \tau_{\geq 1}^{\mathrm{brutal}} M(C_{p^n}/C_{p^i}) = M(C_{p^n}/C_{p^i}) . 
    \] 
\end{definition}

A Mackey functor is brutally truncated if and only if it is a module over $\tau_{\geq 1} \A_{C_{p^n}}$, the brutal truncation of the Burnside Green functor. The Mackey functor $\tau_{\geq 1} \A_{C_{p^n}}$ is the quotient of $\A_{C_{p^n}}$ by the smallest Green ideal containing $\A_{C_{p^n}}(C_{p^n}/e)$ hence is a Green functor. The brutal truncation functor is base-change along the unique Green functor map $\A_{C_{p^n}} \rightarrow \tau_{\geq 1} \A_{C_{p^n}}$.

Define $\mathcal{C}_0$ to be the full subcategory of finitely generated $R$-modules whose underlying Mackey functors are brutally truncated. As long as $\RModfg$ is abelian, $\mathcal{C}_0$ is a Serre subcategory. Thus we may form the localization $\RModfg/\mathcal{C}_0$.

\begin{proposition}\label{prop:C_0-cofiber-identification}
    For any module-Noetherian $C_{p^n}$-Green functor $R$ we have an equivalence of categories
    \[ 
        \RModfg / \mathcal{C}_0 \cong R(C_{p^n}/e)_\theta[C_{p^n}] \textrm{-} \Mod^{\fg}
    \] 
    where on the right hand side $R(C_{p^n}/e)_\theta[C_{p^n}]$ is the twisted group ring in the sense of \cref{definition: twisted group ring}.
\end{proposition}

\begin{proof}
    The exact functor 
    \[ 
        \RModfg \rightarrow R(C_{p^n}/e)_\theta[C_{p^n}] \textrm{-} \Mod^{\fg} 
    \] 
    given by $M \mapsto M(C_{p^n}/e)$ vanishes on $\mathcal{C}_0$, so it is uniquely factored by an additive functor 
    \[ 
        \phi \colon \RModfg / \mathcal{C}_0 \rightarrow R(C_{p^n}/e)_\theta[C_{p^n}] \textrm{-} \Mod^{\fg}.
    \]
    Now let $M$ be an $R(C_{p^n}/e)_\theta[C_{p^n}]$-module. Then $\FP(M)$ is a $R$-module by \cref{lemma: fixed point of twisted module is k-module}. Postcomposing the functor
    \[
        \FP\colon  R(C_{p^n}/e)_\theta[C_{p^n}] \textrm{-} \Mod^{\fg}\to \RModfg
    \]
    with the localization functor defines an additive functor 
    \[ 
        \psi \colon k(C_{p^n}/e)_\theta[C_{p^n}] {-} \Mod^{f.g.} \rightarrow k {-} \Mod^{f.g.} / \mathcal{C}_0 
    \] 
    which we show is a categorical inverse to $\phi$.
    
    We can pick a model for the localization where the objects of $\RModfg / \mathcal{C}_0$ are just the objects of $\RModfg$, and the functor $\phi$ is given on objects by evaluation at $C_{p^n}/e$. The adjunction unit $M \rightarrow \FP(M(C_{p^n}/e))$ has kernel and cokernel in $\mathcal{C}_0$, hence it determines a natural isomorphism $\mathrm{Id} \cong \psi \circ \phi$.

    On the other hand, if we start with a $R(C_{p^n}/e)_\theta[C_{p^n}]$-module $M$, then $\phi \circ \psi(M)$ is obtained by taking the fixed-point Mackey functor to get a $R$-module, regarding it as an object of the localization, and then evaluating at $C_{p^n}/e$. This returns $M$; in particular the fixed-point/evaluation adjunction counit determines the requisite natural isomorphism $\mathrm{Id} \cong \phi \circ \psi$.
\end{proof}

Since any equivalence of abelian categories is automatically an exact equivalence we obtain the following corollary.

\begin{corollary}
    For any module-Noetherian Green functor $R$ there is an equivalence of spectra
    \[
        K(\RModfg / \mathcal{C}_0)\simeq \G(R(C_{p^n}/e)_{\theta}[C_{p^n}]).
    \]
\end{corollary}

By Quillen's localization theorem \cite[Theorem 5]{Quillen:HigherAlgKThy} the sequence of abelian categories  $\mathcal{C}_0\to \RModfg\to R(C_{p^n}/e)_\theta[C_{p^n}] \textrm{-} \Mod^{\fg}$ induces fiber sequences on $K$-theory spectra.  Having identified the $K$-theory of the cofiber, our next task is to identify the spectrum $K(\mathcal{C}_0)$. We do so by relating $\mathcal{C}_0$ to the category of modules over a particular $C_{p^{n-1}}$-Mackey functor and then working inductively.

\begin{definition}
    Let $M$ be a $C_{p^n}$ Mackey or Green functor. We define its \emph{Mackey functor geometric $C_p$-fixed points} $\Phi^{C_p}_{\mathrm{Mackey}} M$ to be the $C_{p^{n-1}}$-Mackey or Green functor obtained from $\tau_{\geq 1} M$ by quotienting the submodule (resp. ideal) given levelwise by the image of the transfer from $M(C_{p^n}/e)$.
\end{definition}

\begin{remark}
    We note that $\Phi^{C_p}_{\mathrm{Mackey}}$ is the $C_p$-geometric fixed points functor considered in \cite{HMQ22}. Thus this functor may equivalently be described as the left adjoint to inflation along the surjection $C_{p^n} \rightarrow C_{p^{n-1}}$. In fact, this inflation may be identified with the inclusion of brutally truncated $C_{p^n}$-Mackey functors into all $C_{p^n}$-Mackey functors. Since this inclusion functor is strong monoidal, it follows formally that the left adjoint $\Phi^{C_p}_{\mathrm{Mackey}}$ is also strong monoidal. In particular, it carries $R$-modules to $\Phi^{C_{p}}_{\mathrm{Mackey}}(R)$-modules for any Green functor $R$. It is known that geometric fixed points preserve both projective and flat modules over the Burnside Green functor; we expect this to be true for modules over any Green functor, although we do not require this result.
\end{remark}

\begin{proposition}\label{prop:C_0-identification}
    For any module-Noetherian Green functor $R$ there are equivalence 
    \[
        \mathcal{C}_0 \cong \Phi^{C_p}_{\mathrm{Mackey}} (R) \textrm{-} \Mod^{\fg}.
    \]
\end{proposition}

\begin{proof}
    It follows immediately from \cref{lemma: definition of a module} that objects and morphisms in these categories are specified by identical data.
\end{proof}

\begin{lemma}\label{lem:Mackey-GFP-preserve-RFD}
    Let $R$ be a relatively finite dimensional $C_{p^n}$-Green functor. Then $\tau_{\geq 1} R$ and $\Phi^{C_p} R$ are relatively finite dimensional.
\end{lemma}

\begin{proof}
    The claim about $\tau_{\geq 1}(R)$ is immediate, as the restriction maps $\tau_{\geq 1}(R)(G/K)\to \tau_{\geq 1}(R)(G/H)$ for $H\leq K$ are the same as the corresponding maps in $R$.  For $\Phi^{C_{p}}(R)$ and $1\leq r< t$ there is a commutative square of rings
    \[
        \begin{tikzcd}
        R(C_{p^n}/C_{p^t}) 
        \ar[->>,r] 
        \ar[d,"\res^{C_{p^t}}_{C_{p^r}}"] 
        & 
        \Phi^{C_{p}}(R)(C_{p^{n-1}}/C_{p^{t-1}}) 
        \ar[d,"\res^{C_{p^{t-1}}}_{C_{p^{r-1}}}"]
        \\
        R(C_{p^n}/C_{p^r}) 
        \ar[->>,r]
        &
        \Phi^{C_{p}}(R)(C_{p^{n-1}}/C_{p^{r-1}})
        \end{tikzcd}
    \]
    and we see the right vertical map is finite because the left vertical map is finite.
\end{proof}

Just as we previously iterated the the functor $\tau_{\geq 1}$ we can iteratively apply the functor $\Phi^{C_p}_{\mathrm{Mackey}}$ $m$-times to obtain $\Phi^{C_{p^m}}_{\mathrm{Mackey}} \colon C_{p^n} \textrm{-} \Mack \to C_{p^{n-m}} \textrm{-} \Mack$.  The following lemma is immediate from the definitions.

\begin{lemma}\label{lem:bottom level of geom fixed points}
    There is an isomorphism of rings with $C_{p^{n-m}}$-action
    \[
        \Phi^{C_{p^m}}_{\mathrm{Mackey}}(R)(C_{p^{n-m}}/e)\cong R(C_{p^{n}}/C_{p^{m}})/\mathrm{im}(\tr^{C_{p^m}}_{C_{p^{m-1}}}).
    \]
\end{lemma}

\begin{notation}
    We will denote either of the (isomorphic) rings in \cref{lem:bottom level of geom fixed points} by $\Phi^{C_{p^m}}(R)$.
\end{notation}

\begin{example}
    For any $R$ we have $\Phi^{e}(R)\cong R(C_{p^n}/e)$.
\end{example}

\begin{proposition}\label{prop:G-theory-tower}
    Let $R$ be a module-Noetherian $C_{p^n}$-Green functor. Then we have a tower of spectra
\[\begin{tikzcd}
	0 \\
	{\G(\Phi^{C_{p^n}}_{\mathrm{Mackey}} R))} & {\G(\Phi^{C_{p^n}}(R))} \\
	{\G(\Phi^{C_{p^{n-1}}}_{\mathrm{Mackey}} (R))} & {\G(\Phi^{C_{p^{n-1}}}(R)_{\theta}[C_{p^n}/C_{p^{n-1}}])} \\
	{\G(\Phi^{C_p}_{\mathrm{Mackey}} (R))} & {\G(\Phi^{C_p} (R)_\theta[C_{p^n}/C_p])} \\
	{\G(R)} & {\G(\Phi^{C_{p^0}}(R)_\theta[C_{p^n}/e])}
	\arrow[from=1-1, to=2-1]
	\arrow["{=}", from=2-1, to=2-2]
	\arrow[from=2-1, to=3-1]
	\arrow[from=3-1, to=3-2]
	\arrow[dashed, from=3-1, to=4-1]
	\arrow[from=4-1, to=4-2]
	\arrow[from=4-1, to=5-1]
	\arrow[from=5-1, to=5-2]
\end{tikzcd}\]
    where the cofiber of each vertical map is the corresponding horizontal map.
\end{proposition}

\begin{proof}
    For any $0\leq m\leq {n-1}$ let $\mathcal{C}_m\subset \RModfg$ denote the Serre subcategory of finitely generated $R$-modules $M$ such that $M(C_{p^n}/C_{p^t})=0$ for all $t\leq m$.  This gives a filtration 
    \begin{equation}\label{equation:filtration of module cateogry}
        0\hookrightarrow \mathcal{C}_n\hookrightarrow \mathcal{C}_{n-1}\hookrightarrow\dots\hookrightarrow\mathcal{C}_{0}\hookrightarrow \kModfg.
    \end{equation}
    where each functor is the inclusion of a Serre subcategory.
    
    Through the identification of \cref{prop:C_0-identification} we see that $\mathcal{C}_1$ is equivalent to the collection of brutally truncated $\Phi^{C_p}_{\mathrm{Mackey}}(R)$-modules, hence we can apply \cref{prop:C_0-identification} again to obtain an equivalence of abelian categories
    \[
        \mathcal{C}_1\simeq \Phi^{C_{p^2}}_{\mathrm{Mackey}}(R)\textrm{-}\Mod^{\fg}.
    \]
    By induction, we have $\mathcal{C}_m\simeq \Phi^{C_{p^{m-1}}}_{\mathrm{Mackey}}(R)\textrm{-}\Mod^{\fg}$ and thus taking $K$-theory of the chain \eqref{equation:filtration of module cateogry}, together with Quillen's localization theorem \cite[Theorem 5]{Quillen:HigherAlgKThy}, yields the tower of spectra from the statement.  The identification of the cofibers is \cref{prop:C_0-cofiber-identification}.
\end{proof}

\begin{remark}\label{remark: SS for general G}
    For an arbitrary finite group $G$ and $G$-Green functor $R$ it is still possible to filter the category of finitely generated $R$-modules ``by isotropy'' as we have done here and thus produce a filtration of $\G(R)$. It is less clear that pieces of the associated graded have a recognizable form, although we certainly believe this approach to be feasible. We do not pursue this idea further in this paper.
\end{remark}

\begin{corollary}\label{cor:G-theory-SS}
    For any module Noetherian $C_{p^n}$-Green functor $R$ there is a strongly convergent spectral sequence with signature 
    \[
        E^1_{s,t} \cong \G_s(\Phi^{C_{p^{t}}}(R)_\theta[C_{p^n}/C_{p^t}]) \Rightarrow \G_{s}(R)
    \] 
    with differentials
    \[
        d_r : E^r_{s,t} \rightarrow E^r_{s-1,t+r}.
    \] 
    The spectral sequence collapses at the $E_{n+1}$ page.
\end{corollary}

\begin{proof}
    There is an exact couple
    \[
        E^1_{x,y} = \pi_{x+y}\G(\Phi^{C_{p^{n-x}}}(R)_{\theta}[C_{p^n}/C_{p^{n-x}}])\quad D^1_{x,y} = \pi_{x+y}\G(\Phi^{C_{p^{n-x}}}_{\mathrm{Mackey}}(R))
    \]
    obtained from the tower of cofibrations in \cref{prop:G-theory-tower} where $d^r$ has bidegree $(-r,r-1)$. Here we interpret any group with $x<0$ or $x>n$ to be zero. This spectral sequence abuts to $\G_{x+y}(R)$, and since the values of $x$ are bounded this spectral sequence collapses at page $r=n+1$ and thus is strongly convergent. The grading in the statement results from the choices $s= x+y$ and $t= n-x$.
\end{proof}

The spectral sequence is visualized in \cref{table:G-theory-SS-E_1-page}. As an immediate corollary we can compute the $\G$-theory of any Green functor with surjective transfers (such as Green meadows whose characteristic is not $p$, or for which $C_{p^n}$ acts faithfully on the $C_{p^n}/e$-level).

\begin{proposition}\label{prop:Greenlees--May-for-G-theory-when-transfers-surjective}
    If $R$ is a module-Noetherian $C_{p^n}$-Green functor in which all transfers are surjective then $\G(R)\simeq \G(R(C_{p^n}/e)_{\theta}[C_{p^n}])$.
\end{proposition}

\begin{proof}
    The surjectivity of the transfers implies that $E^1_{s,t}=0$ for $t>0$ in the $\G$-theory spectral sequence.  Thus the spectral sequence has a single non-trivial horizontal line so it collapses at $E^1$ and there are no possible extensions.
\end{proof}

\begin{corollary}\label{cor: G theory of faithful fixed point}
    If $L$ is a field with faithful $C_{p^n}$-action then $\G(\FP(L)) \simeq K(L^{C_{p^n}})$.
\end{corollary}

\begin{proof}
    The transfer map in $\FP(L)$ is the Galois transfer.  Since the action of $G$ on $L$ is faithful it is Galois, by Artin's lemma, and thus the transfers are all surjective and the proposition gives us $\G(\FP(L))\simeq \G(L_{\theta}[C_{p^n}])$. Since $L_{\theta}[C_{p^n}]$ is Morita equivalent to $L^{C_{p^n}}$ by \cref{lem:twisted-group-ring-Morita-invariance} below, we have $\G(\FP(L)) \simeq \G(L^{C_{p^n}})= K(L^{C_{p^n}})$ where the second equivalence uses that $L^{C_{p^n}}$ is a field.
\end{proof}


\begin{figure}[h]
\centering

\begin{sseqpage}[grid = crossword, 
    classes = {draw = none}, Adams grading,
    title={},
    xrange = {0}{4},
    yrange = {0}{4},
    x label = $s$, 
    y label = $t$,
    xscale =2,
    y axis gap=25pt]

\foreach \s in {0,1,2,3} {
\foreach \t in {0,1,2,3} {
    \class["\G_{\s}(R^{\Phi \t}_{\theta})"](\s,\t)
    }
}
\foreach \p in {0,1,2,3} {
\class["\dots"](4,\p)
}
\end{sseqpage}

\caption{The $E^1$-page of the $\G$-theory spectral sequence for a $C_{p^3}$-Green functor $R$. For brevity, we have written $R^{\Phi t}_{\theta}$ for the ring $\Phi^{C_{p^{t}}}(R)_{\theta}[C_{p^{3-t}}]$.\label{table:G-theory-SS-E_1-page}} 

\end{figure}

We now turn our attention to the behavior of the $\G$-theory tower in more complicated examples. In special cases, the connecting maps turn out to be summand inclusions. The algebraic consequences of this are captured in the following definition.

\begin{definition}[Greenlees--May splitting]\label{def:G-theory-splitting}
        Let $R$ be a $G$-Green functor and write the subgroups of $G$ sequentially as $\{ H_i \}$ with the property $H_i \subset H_j$ implies $i \leq j$. 
        
        If there is an equivalence 
        \[ 
            \G(R) \cong \bigoplus_p \G(\Phi^{H_p}(R)_\theta[W_G H_p]) 
        \] 
        of $\G$-theory spectra, then we say $R$ satisfies the \emph{$\G$-theory Greenlees--May splitting}.
\end{definition}

\begin{remark}
    Fix a field $\F$, considered as a ring with trivial $G$-action, and let $R$ be the Green functor obtained from the Burnside by base changing along $\mathbb{Z} \to \F$, levelwise. An $R$-module is the same data as a Mackey functor in $\F$-modules (not necessarily cohomological). Thus \cite[Proposition 6.1]{Thevenaz--Webb} says that $\pi_0$ of both sides of the Greenlees--May splitting are isomorphic. In other words, $R$ satisfies the Greenlees--May splitting for $\G_0$. For $G = C_{p^n}$, generalize this to a splitting of the Burnside itself in \cref{theorem: G theory of Burnside} below.
\end{remark}

\cref{prop:Greenlees--May-for-G-theory-when-transfers-surjective} provides an example of the Greenlees--May splitting for $\G$-theory. On the other hand, \cref{prop:K_0(FP(Z))} below provides an example of the failure of the Greenlees--May splitting for $\G$-theory (as the rank of $\G_0(\underline{\Z})$ is one less than what is predicted by the Greenlees--May splitting).  The next theorem provides a sufficient condition on a $C_{p^n}$-Green functor to guarantee a $\G$-theory Greenlees--May splitting.

\begin{theorem}\label{thm:G-theory-splitting}
    Let $R$ be a module-Noetherian $C_{p^n}$-Green functor, and suppose each quotient $q \colon \tau_{\geq i}(R) \rightarrow \Phi^{C_{p^i}}_{\mathrm{Mackey}}(R)$ admits a module-finite section $s\colon \Phi^{C_{p^i}}_{\mathrm{Mackey}}(R)\to \tau_{\geq i}(R)$. Then the tower of \cref{prop:G-theory-tower} splits (in the sense that the connecting morphisms are summand inclusions), so that $R$ satisfies the Greenlees--May splitting for $\G$-theory.
\end{theorem}

\begin{proof}
    Our hypotheses imply that each quotient $\tau_{\geq 1} \Phi^{C_{p^{i-1}}} (R) \rightarrow \Phi^{C_p} \Phi^{C_{p^{i-1}}}(R)$ admits a section, so by induction it suffices to assume $i = 1$ and construct an exact functor $\pi\colon \RModfg\rightarrow \mathcal{C}_0$, such that $\pi\circ f\cong \mathrm{id}$. Consider the commutative diagram
\[\begin{tikzcd}
	{R\textrm{-}\Mod} \\
	& {R\textrm{-}\Mod^{\geq 1}} & {\tau_{\geq 1}(R)\textrm{-}\Mod} \\
	{\mathcal{C}_0} && {\Phi^{C_{p}}_{\mathrm{Mackey}}(R)\textrm{-}\Mod}
	\arrow["{\tau_{\geq 1}}", from=1-1, to=2-3]
	\arrow["g", from=2-2, to=1-1]
	\arrow["\varphi"', from=2-2, to=2-3]
	\arrow["{s^*}", shift left, from=2-3, to=3-3]
	\arrow["f", from=3-1, to=1-1]
	\arrow["\psi"', from=3-1, to=3-3]
	\arrow["{q^*}", shift left, from=3-3, to=2-3]
\end{tikzcd}\]
where $f$ and $g$ are inclusions, $\psi$ is the equivalence of \cref{prop:C_0-identification}, and $\varphi$ is the equivalence from \cref{prop:equivalence-of-cats-tau}.  Let $\alpha$ be a categorical inverse to $\psi$.  The commutativity of the diagram, together with the fact that $s^*\circ q^* = (q\circ s)^* =1$, implies that 
\[
    (\alpha\circ s^*\circ \tau_{\geq 1})\circ f\cong \alpha \circ s^*\circ q^*\circ\psi\cong \mathrm{id}.
\]
Since $\alpha$, $s^*$, and $\tau_{\geq 1}$ are exact we have that $\pi = \alpha\circ s^*\circ \tau_{\geq 1}$ is also exact.

The additional assumption that $s^*$ is module-finite implies that $s^*$ preserves finitely generated modules implies that $\pi$ restricts to a functor on finitely generated $R$-modules, hence produces a splitting of the fiber sequences in the $\G$-theory tower, proving the claim.
\end{proof}

\begin{corollary}\label{corollary: zero transfer splitting}
    Suppose $R$ is module-Noetherian and that all transfers in $R$ are zero. Then there is an equivalence 
    \[
        \G(k) \simeq \bigoplus_{i=0}^n \G(k(C_{p^n}/C_{p^i})_\theta[C_{p^n}/C_{p^i}]).
    \]
\end{corollary}

\begin{proof}
    It suffices to verify the hypothesis of \cref{thm:G-theory-splitting}. In fact, our assumption implies that the quotient $\tau_{\geq i} R \rightarrow \Phi^{C_{p^i}} R$ is an isomorphism. Its inverse provides the required section. Note that we have used that any isomorphism is module-finite.
\end{proof}

We now have enough to compute the $\G$-theory of every relatively finite dimensional $C_p$-Green meadow. 

\begin{theorem}\label{theorem: G theory of Cp Tambara fields}
    Let $k$ be any relatively finite dimensional $C_p$-Green meadow. If the transfer map is identically zero then
    \[
        \G(k)\simeq \G(k(C_{p}/C_p))\oplus \G(k(C_p/e)_{\theta}[C_{p}])
    \]
    and if the transfer is non-zero then
    \[
        \G(k)\simeq \G(k(C_p/e)_{\theta}[C_{p}])\simeq K(k(C_p/e)).
    \]
\end{theorem}

In particular, every relatively finite dimensional $C_p$-Green meadow satisfies the $\G$-theory Greenlees--May splitting.

\begin{proof}
    If the transfer is zero then the result follows at once from \cref{corollary: zero transfer splitting}.  If the transfer is not zero then because the image of the transfer map is an ideal of $k(C_p/C_p)$, which is a field, the transfer must be surjective.  Thus the non-zero portion of $\G$-theory spectral sequence consists of a single horizontal line and collapses at the $E^1$-page.
\end{proof}

The next theorem gives an example where not all the transfers are zero, and generalizes a result of Greenlees \cite{Gre92}, who proved the analogous result on $\G_0$ when $n=1$.

\begin{theorem}\label{theorem: G theory of Burnside}
    There is an equivalence of spectra
    \[
        \G(\mathrm{Mack}_{C_{p^n}}^{f.g.}) \simeq \bigoplus\limits_{t=0}^n \G(\mathbb{Z}[C_{p^{n-t}}]).
    \]
\end{theorem}

\begin{proof}
    The category of finitely generated Mackey functors is the same as the category of finitely generated modules over the Burnside Green functor $\A_G$, so the result follows from showing the the $C_{p^n}$-Burnside Green functor satisfies the hypotheses of \cref{thm:G-theory-splitting}.

    Since $\Phi^{C_p}_{\mathrm{Mackey}}$ is symmetric monoidal and the Burnside Green functor is the symmetric monoidal unit, we deduce $\Phi^{C_p}_{\mathrm{Mackey}}(\A_{C_{p^n}}) \cong \A_{C_{p^{n-1}}}$ as Green functors. Since $\tau_{\geq 1} \A_{C_{p^n}}$ is also a Green functor, the unit $\A_{C_{p^{n-1}}} \rightarrow \tau_{\geq 1} \A_{C_{p^n}}$ provides the required section. Since each ring $\A_G(G/H)$ is a finitely generated torsion-free $\Z$-module, all Green functors in sight are relatively finite dimensional and our section is module-finite.
\end{proof}

\begin{remark}
    Let $A$ and $B$ be ring so that there exists a $C_p$-Green functor $R$ with $R(C_p/e)_\theta[C_p] \cong B$ and $\Phi^{C_p}(R) \cong A$. Then the $d^1$ differential in the $\G$-theory spectral sequence gives a group homomorphism $\G_{i+1}(B) \rightarrow \G_i(A)$ for all $i \geq 0$. Alternatively, since $G = C_p$, the $\G$-theory tower is just a cofiber sequence, and this is just the boundary map. So, either we have a nontrivial morphism between $\G$-theories, or the spectral sequence collapses. 
    
    Either outcome is interesting, although we can only determine which is true in special cases. For example, \cref{thm:G-theory-splitting} provides many examples in which all differentials are trivial. On the other hand, we will see an example of a nontrivial differential in the proof of \cref{prop:K_0(FP(Z))}.
\end{remark}

\begin{example}
    Fix a ring $A$ of characteristic $p$, and consider the category $\mathcal{C}$ of cohomological $C_{p^n}$-Mackey functor objects in the category of finitely generated $R$-modules. The category $\mathcal{C}$ is equivalent to the category of finitely generated $\underline{\F}$-modules, and is, by a theorem of Yoshida, equivalent to the category of finitely generated modules over a certain Hecke algebra $\mathcal{H}$ \cite{Yos83}. \cref{corollary: zero transfer splitting} then applies to compute
    \[ 
        \G(\mathcal{H}) \cong K(\mathcal{C}) \cong \G(\FP(R)) \cong \bigoplus_{i=0}^n \G(R[C_{p^{n-i}}]) \mathrm{.}
    \]
\end{example}

\begin{example}
    Let $R$ be the $C_p$-Green functor given by the Lewis diagram
    \[ 
    \begin{tikzcd}
        \F_p[t]/(t^2) \arrow[d, "r"] \\
        \F_p \arrow[u, bend left = 55, "t"] \arrow[loop below, "\mathrm{triv}"]
    \end{tikzcd} 
    \]
    with restriction and transfer obtained from the Burnside $C_p$-Green functor by base-changing along $\Z \rightarrow \F_p$. Then the category of $R$-modules may be identified with the category of $C_p$-Mackey functor objects in $\F_p$-vector spaces. There is a map $R \rightarrow \underline{\F_p}$, and modules over $\underline{\F_p}$ may be identified with the full subcategory of the category of $R$-modules spanned by those which are cohomological.

    Now $\Phi^{C_p} R \cong \F_p$, so that $R$ visibly satisfies the hypothesis of \cref{thm:G-theory-splitting}. In particular, we have an equivalence
    \[
        \G(R) \cong \G(\F_p) \oplus \G(\F_p[C_p]) \cong \G(\underline{\F_p}) .
    \]
    Thus there is no $\G$-theory obstruction to a Morita equivalence between $R$ and $\underline{\F_p}$. Despite this, it seems like $R$ and $\underline{\F_p}$ are most likely not Morita equivalent.
\end{example}

\subsection{\texorpdfstring{Towards $K$-theory}{Towards K-theory}}

In \cite{Gre92}, Greenlees proves that a version of the Greenlees--May splitting holds for the $K_0$-groups of some Green functors.  In this section, we make a few remarks on extending this sort of computation to more Green functors and higher algebraic $K$-groups.  Similar to $\G$-theory, we organize these results around the desired behavior.

\begin{definition}[Greenlees--May splitting]\label{def:K-theory-splitting}
        Let $R$ be a $G$-Green functor. If there is an equivalence 
        \[ 
            K(R) \cong \bigoplus_p K(\Phi^{H_p}(R)_\theta[W_G H_p]) 
        \] 
        of $K$-theory spectra, then we say $R$ satisfies the \emph{$K$-theory Greenlees--May splitting}.
\end{definition}

\begin{remark}
    In the next section we show that not every $C_2$-Green functor satisfies the $K$-theory Greenlees--May splitting.
\end{remark}

One might hope to form a $K$-theory spectral sequence via a filtration on the exact category of finite free $R$-modules with $t$-th filtered piece the subcategory of free $R$-modules $F_i$ with $i \geq t$ (recalling that the subcategory of free modules is sufficient to compute all higher $K$-groups except $K_0$). The issue with this approach is that it is unclear what the corresponding Verdier quotient is.

In spite of this, we are still able to establish a few cases of Greenlees--May splitting for $K$-theory. We begin with the following well-known result.

\begin{lemma}\label{lem:twisted-group-ring-Morita-invariance}
    Let $G$ act faithfully on a field $\F$. Then the inclusion $\F^G \rightarrow \F_\theta[G]$ is a Morita equivalence.
\end{lemma}

\begin{proof}
    We are in the situation of \cite[Example 28.3]{CR81}. In particular, the map $\F_\theta[G] \rightarrow \mathrm{End}_{\F^G} \F$ describing the left $\F_\theta[G]$ module structure on $\F$ is an isomorphism.
\end{proof}

In particular, the inclusion of \cref{lem:twisted-group-ring-Morita-invariance} induces an equivalence on $\G$-theory and $K$-theory spectra.  Of course, since every module over a field is free then relevant $\G$- and $K$-theory spectra are identical.

\begin{proposition}\label{prop:Greenlees--May splitting for C_p in char p with nontrivial action}
    Let $k$ be a $C_p$-Green meadow of characteristic $p$ such that $C_p$ acts nontrivially on $k(C_p/e)$. Then $k$ satisfies the Greenlees--May splitting for $K$-theory.
\end{proposition}

\begin{proof}
    The transfer is surjective, so that the $C_p$-geometric fixed points vanish, and the Greenlees--May splitting amounts to showing
    \[
        K(k) \simeq K(k(C_p/e)_\theta[C_p]) .
    \]
    We will show that the category of $k$-modules is equivalent to the category of $k(C_p/e)_\theta[C_p]$-modules. 
    
    If $M$ is an $k$-module, we see from \cref{lem:twisted-group-ring-Morita-invariance} that $M(C_p/e)$ is ($C_p$-equivariantly) a direct sum of copies of $k(C_p/e)$. Frobenius reciprocity implies that the restriction in $M$ is injective, and the double coset formula implies furthermore that $M$ is a fixed-point Mackey functor. Thus the right adjoint $\FP$ from $k(C_p/e)_\theta[C_p]$-modules to $k$-modules is essentially surjective. With this, both the adjunction unit and the adjunction counit for $\FP$ are clearly isomorphisms, so that $\FP$ is an equivalence of categories.
\end{proof}

The next result could also be obtained using a recent result of Bouc--Dell'Amborgio--Martos \cite[Theorem A]{BoucDellAmbrogioMartos}.

\begin{proposition}
    Let $k$ be a $G$-Green meadow of characteristic either zero or coprime to $|G|$. Then $k$ satisfies the Greenlees--May splitting for $K$-theory.
\end{proposition}
\begin{proof}
    The assumptions imply that $k$ is fixed-point Green functor $k \cong \FP(\F)$ for some $G$-action on an honest field $\F$. The unit $\Z \rightarrow \F$ determines a map $\FP(\Z) \rightarrow k$, so that every $k$-module is also a $\FP(\Z)$-module, i.e. is cohomological. This, plus the characteristic assumption, implies that every $k$-module is the fixed-point module associated to a $k(G/e)_\theta[G]$-module. In particular, the functor 
    \[ 
        \FP : k(G/e)_\theta[G] {-} \Mod \rightarrow k {-} \Mod 
    \] 
    is an additive equivalence of categories which preserves separately the properties of being free (hence projective) and of being finitely generated.

    The characteristic assumption and the fact that $k$ is cohomological imply that all transfers in $k$ are surjective. In particular, whenever $H \neq e$, the $H$-geometric fixed-points of $k$ vanish. Thus the only nonzero term in the Greenlees--May splitting is $K(k(G/e)_\theta[G])$, so that the equivalence of $K$-theory spectra induced by $\FP$ implies the claim.
\end{proof}

In the course of the proof, we observed the following:

\begin{corollary}\label{cor:nonmodular-characteristic-Morita-equiv}
    Let $k$ be a $G$-Green meadow of characteristic either zero or coprime to $|G|$. Then $k$ is Morita equivalent to the twisted group ring $k(G/e)_\theta[G]$.
\end{corollary}

The authors expect the following to be true: if an $H$-Green functor $R$ satisfies the Greenlees--May splitting, then so does (the Green functor structure on the Mackey functor) $\Ind_H^G R$. Using \cite[Corollary B]{Wis25}, a proof of this result amounts to establishing Morita equivalences of appropriate twisted group rings. If, alternatively, the expected twisted group ring Morita equivalences failed to hold, then induction would supply more cases of the failure of the Greenlees--May splitting.

%% file: applications.tex
\section{\texorpdfstring{$K$-theory computations}{K-theory computations}}\label{section: computations}

In this section we give some explicit computations of $K$-theory and $\G$-theory as the payoff of the machinery developed in previous sections.

\subsection{The $K$-theory of \texorpdfstring{$\underline{\F_2}$}{constant F\_2}}

In this subsection we consider the $K$-theory of the constant $C_2$-Green functor $\underline{\F_2}$. Our main result is a consequence of a much more general observation.

\begin{theorem}\label{thm:K-theory-splitting-for-regular-case}
    Let $R$ be a $C_{p^n}$-Green functor satisfying the hypotheses of \cref{thm:G-theory-splitting} and assume that the category of finitely generated $R$-modules has finite global projective dimension. Then we have an equivalence
    \[
        K(R) \cong \bigoplus_{i=0}^n \G(\Phi^{C_{p^i}} R_\theta[C_{p^n}/C_{p^i}])
    \]
    of spectra.
\end{theorem}

\begin{proof}
    The finite global projective dimension hypothesis implies, by the resolution theorem, the equivalence $K(R) \simeq \G(R)$ \cite[\S 4, Corollary 2]{Quillen:HigherAlgKThy}. The result then follows from the $\G$-theory splitting \cref{thm:G-theory-splitting}.
\end{proof}

\begin{theorem}\label{thm:K-theory-of-FP-F_2}
  There is a splitting of $K$-theory spectra
    \[
        K(\underline{\F_2}) \cong K(\F_2) \oplus K(\F_2)
    \]
    in the stable homotopy category.
\end{theorem}

\begin{proof}
    The Green functor $\underline{\F_2}$ satisfies the hypotheses of \cref{thm:G-theory-splitting} since all transfers are zero. Additionally, by \cite[Proposition 3.10]{DHM24}, we see that \cref{thm:K-theory-splitting-for-regular-case} applies, so that we obtain 
    \[
        K(\underline{\F_2}) \cong K(\F_2) \oplus \G(\F_2[C_2]) \mathrm{.}
    \]
    Now $\F_2[C_2]$ is a local Artin ring with residue field $\F_2$, which implies $\G(\F_2[C_2]) \cong K(\F_2)$.
\end{proof}

\begin{corollary}\label{cor:K-groups-of-FP-F_2}
    The $K$-groups of $\underline{\F_2}$ are given as follows.
    \[
        K_{n}(\underline{\F_2}) \cong \begin{cases}
            A(C_2) & n=0\\
            (\mathbb{Z}/(2^i-1))^{\oplus 2} & n = 2i-1\\
            0 & \mathrm{else}.
        \end{cases}
    \]
    In particular, $\underline{\F_2}$ does not satisfy the Greenlees--May splitting for $K$-theory.
\end{corollary}

\begin{proof}
    The computation follows from Quillen's computation of $K_n(\F_p)$ for all $n \geq 0$ \cite[Theorem 8(iii)]{Qui72}. To see that the Greenlees--May splitting for $K$-theory fails, observe that
    \[
        K_1(\underline{\F_2}) \cong 0
    \]
    whereas the right-hand side of the equation predicted by the Greenlees--May splitting has $\pi_1$ isomorphic to 
    \[ 
        K_1(\F_2) \oplus K_1(\F_2[C_2]) \cong K_1(\F_2[C_2]) \mathrm{.}
    \]
    Since $K_1(\F_2[C_2])\cong \Z/2\Z$ the Greenlees--May splitting must fail.
\end{proof}

Interestingly, the Greenlees--May splitting \emph{does} hold after passing to $K_0$. In particular, \cref{thm:K_0-of-clarified-triv-action-Cpn-Tamb-fields} implies $K_0(\underline{\F_2})$ is the free abelian group generated by $\underline{\F_2}$ and $\Ind_e^{C_2} \F_2$. These respectively correspond to the generators for $K_0(\F_2) \cong \Z$ and $K_0(\F_2[C_2]) \cong \Z$.

\begin{remark}\label{remark: p odd has infinite global dimension}
    It is reasonable to wonder whether a similar approach could be used to compute the $K$-theory of $\underline{\F_p}$ when $G = C_p$ for odd primes, or $\underline{\F_2}$ when $G = C_4$. This approach does not work as, in either case, the relevant category of modules does not have finite global projective dimension, and \cref{thm:K-theory-splitting-for-regular-case} does not apply. 
\end{remark}

\subsection{The \texorpdfstring{$K$}{K}-theory of \texorpdfstring{$\underline{\Z}$}{constant Z}}

We recall one of the main results of \cite{BSW17}. In particular, the following is obtained from \cite[Theorem 1.6]{BSW17} by observing that the category of cohomological $G$-Mackey functors is identified with the category of modules over $\underline{\Z}$, giving $\Z$ the trivial $G$-action.

\begin{theorem}[{\cite[Theorem 1.7]{BSW17}}]\label{theorem:BoucStancuWebb}
    The $G$-Green functor $\underline{\Z}$ has finite global projective dimension if and only if, for every odd prime $p$, the $p$-Sylow subgroups of $G$ are cyclic, and the $2$-Sylow subgroups of $G$ are cyclic or dihedral.
\end{theorem}

When $p$ is odd, $C_{p^n}$ satisfies the hypothesis of this theorem, although when $p=2$, the theorem applies only to $C_2$, and not to $C_{2^n}$ for $n>1$. Indeed, when $G = C_4$, it is observed in \cite{BSW17} that $\underline{\Z}$ does not have finite global projective dimension.

\begin{corollary}\label{cor:FP-Z-is-regular}
    Let $G = C_{p^n}$. We have an equivalence 
    \[ 
        \G(\underline{\Z}) \simeq K(\underline{\Z})
    \]
    of spectra. The spectral sequence of \cref{cor:G-theory-SS} thus converges to the $K$-theory of $\underline{\Z}$.
\end{corollary}

The first statement of \cref{cor:FP-Z-is-regular} is true for any $G$ satisfying the hypotheses of \cite[Theorem 1.6]{BSW17}. The second statement should be true, provided one constructs the necessary filtration on $\G(R)$; see \cref{remark: SS for general G}.

\begin{remark}
    
    
    The composite
    \[
        \mathrm{Ab}\xrightarrow{\mathrm{triv}} \Z[G] \textrm{-}\Mod\xrightarrow{\FP} \underline{\Z} \textrm{-}\Mod
    \]
    is exact and preserves projectives.  Moreover it is an exact section of the exact functor $M \mapsto M(G/e)$. In particular, we see that $K(\Z)$ is a retract of $K(\underline{\Z})$, for any finite group $G$. Since the full computation of $K_n(\Z)$ is an extremely difficult open problem, we should not expect to be able to compute all $K$-groups of $\underline{\Z}$.
\end{remark}

We can now state and prove \cref{introthm: K of constant Z}.

\begin{theorem}\label{thm:K-theory-of-FP(Z)}
    Let $G = C_{p^n}$ for $p$ a prime. There is a map 
    \[
        K(\underline{\Z}) \rightarrow \G(\Z[C_{p^n}])
    \]
    which becomes an equivalence after applying $p$-completion and taking $2$-connective covers. Precisely, we have 
    \[
        \tau_{\geq 2} \left( K(\underline{\Z})_{p}^{\wedge} \right) \cong \tau_{\geq 2} \left( \G(\Z[C_{p^n}])^{\wedge}_p \right) .
    \]
\end{theorem}

\begin{proof}
    By \cref{cor:FP-Z-is-regular} we have an equivalence $K(\underline{\Z}) \cong \G(\underline{\Z})$. We consider the finite filtration of $\G(\underline{\Z})$ supplied by \cref{prop:G-theory-tower}. Every associated graded piece, except for $\G(\Z[C_{p^n}])$, is equivalent to $\G(\F_p[C_{p^i}])$ for some $i>0$. Since $\F_p[C_{p^i}]$ is a local Artin ring with residue field $\F_p$, we obtain $\G(\F_p[C_{p^i}]) \cong K(\F_p)$. By Quillen's computation of the $K$-theory of finite fields \cite{Qui72}, we have an identification 
    \[
        K(\F_p)_{p}^{\wedge} \cong H\Z_p.
    \]
    Since $p$-completion preserves cofiber sequences, $p$-completion of the $\G$-theory tower of \cref{prop:G-theory-tower} becomes a tower for the $p$-completion $\G(\underline{\Z})$ as it has finitely many nonzero associated graded pieces. In the associated spectral sequence, we thus observe that all classes in topological degree at least $2$ are permanent cycles. The part of the $E_2$-page in topological degree at least $2$ is precisely the horizontal zero line $\G_i(\Z[C_{p^n}])$ for $i \geq 2$. Thus the induced maps $K_i(\underline{\Z}) \rightarrow \G_i(\Z[C_{p^n}])$ are isomorphisms for all $i \geq 2$, as desired.
\end{proof}

Now we specialize to the case $G = C_p$ and give some low degree computations before $p$-completion. The $C_p$-geometric fixed-points of $\underline{\Z}$ are $\F_p$ and $\G(\F_p)=K(\F_p)$ since $\F_p$ is a field. There is no section to the quotient $\Z \rightarrow \F_p$, so \cref{thm:G-theory-splitting} does not apply. However, since the $\G$-theory filtration only has two nontrivial associated graded pieces, the $\G$-theory spectral sequence is the same data as the long exact sequence associated to the cofiber sequence
\[
    K(\F_p) \rightarrow K(\underline{\Z}) \rightarrow \G(\Z[C_p])
\]
from which some information may be extracted.

The decomposition of $\G_1$ of integral group rings was given by Keating in \cite{Kea76}. A decomposition for the $\G$-groups of integral group rings $\Z[\pi]$, for $\pi$ abelian, was given by Webb \cite{Webb:QuillenGTheory}. In particular, we have 
\[ 
    \G_n(\Z[C_p]) \cong \G_n(\Z) \oplus \G_n(\Z[p^{-1},\zeta]) 
\]
where $\zeta$ is a primitive $p$-th root of unity. Since $\Z[\zeta]$ is the ring of algebraic integers in $\Q(\zeta)$ it is a Dedekind domain. In particular, it is hereditary (every submodule of a projective module is projective), and since localizations of hereditary rings are hereditary, we see that $\Z[p^{-1},\zeta]$ is also hereditary. Since hereditary implies that global projective dimension is at most $1$, $\Z[p^{-1},\zeta]$ is regular. Since $\Z$ is also regular we obtain
\[ 
    \G_n(\Z[C_p]) \cong K_n(\Z) \oplus K_n(\Z[p^{-1}][\zeta]) \mathrm{.}
\]
The long exact sequence induced by the $\G$-theory filtration and our various identifications then becomes a long exact of $K$-groups, by the miraculous fact that everything in sight has finite global projective dimension:
\[
    \hdots \rightarrow K_i(\F_p) \rightarrow K_i(\underline{\Z}) \rightarrow K_i(\Z) \oplus K_i(\Z[p^{-1},\zeta]) \rightarrow K_{i-1}(\F_p) \rightarrow \hdots 
\]
By examining low degrees we may deduce the following.

\begin{proposition}\label{prop:K_0(FP(Z))}
    Let $G = C_p$. Then we have an isomorphism 
    \[ 
        K_0(\underline{\Z}) \cong \Z \oplus K_0(\Z[p^{-1},\zeta])
    \] 
    of abelian groups.
\end{proposition}

We note that a very special case of \cite[Theorem 16.5]{Thevenaz--Webb} gives a classification of projective $\underline{\Z}$-modules. Although this does not immediately yield a computation of $K_0(\underline{\Z})$, it could provide an alternate path to proving \cref{prop:K_0(FP(Z))}.

\begin{proof}
    In low degrees we obtain the exact sequence
    \[ 
        K_0(\F_p) \rightarrow K_0(\underline{\Z}) \rightarrow K_0(\Z) \oplus K_0(\Z[p^{-1},\zeta]) \rightarrow 0 \mathrm{.}
    \]
    The map $K_0(\F_p) \rightarrow K_0(\underline{\Z})$ is really the map
    \[
        \G_0(\F_p)\to \G_0(\underline{\Z})
    \]
    and is determined by what it does to the class of the projective $\F_p$-module $\F_p$. Consider the $\underline{\Z}$-module $M$ which is $\F_p$ in the $C_p/C_p$-level and $0$ in the $C_p/e$-level; this has a finite free resolution 
    \[ 
        0 \rightarrow \underline{\Z} \rightarrow \Ind_e^{C_p} \Z \rightarrow \Ind_e^{C_p} \Z \rightarrow \underline{\Z} \rightarrow M \rightarrow 0 .
    \]
    The resolution theorem identifying $\G$-theory with $K$-theory implies that our map is given by sending $[\F_p]$ to $[\underline{\Z}] - [\Ind_e^{C_p} \Z] + [\Ind_e^{C_p} \Z] - [\underline{\Z}] = 0$. So $K_0(\underline{\Z}) \rightarrow K_0(\Z) \oplus K_0(\Z[p^{-1},\zeta])$ is an isomorphism.
\end{proof}

\begin{example}
    The group $K_0(\Z[p^{-1},\zeta])$ is the quotient of $K_0(\Z[\zeta])$ by the projective modules which vanish after $p$ is inverted. The group $K_0(\Z[\zeta])$ is isomorphic to the direct sum of $\Z$ and the ideal class group of the cyclotomic extension $\Q[\zeta]$. In particular, when $p<23$ the group $K_0(\Z[p^{-1},\zeta])$ is isomorphic to $\Z$, so that $K_0(\underline{\Z}) \cong \Z^2$.
\end{example}

Finally, we turn our attention to $K_1(\underline{\Z})$. From the proof of \cref{prop:K_0(FP(Z))} (and some identifications of well-known $K$-theory groups), we obtain the exact sequence 
\begin{align*}
    0 & \rightarrow K_2(\underline{\Z}) \rightarrow K_2(\Z) \oplus K_2(\Z[p^{-1},\zeta]) \rightarrow \Z/(p-1) \\
    & \rightarrow K_1(\underline{\Z}) \rightarrow \Z/2 \oplus K_1(\Z[p^{-1},\zeta]) \rightarrow \Z \rightarrow 0 \mathrm{.}
\end{align*}

From \cref{thm:K-theory-of-FP(Z)} and our discussion so far, the following is immediate.

\begin{proposition}
    For $i > 1$, the rank of $K_i(\underline{\Z})$ is equal to the rank of $K_i(\Z[C_p])$. The rank of $K_1(\underline{\Z})$ is one less than the rank of $K_1(\Z[p^{-1},\zeta])$.
\end{proposition}